\documentclass{amsart}

\usepackage[utf8]{inputenc}
\usepackage[T1]{fontenc}
\usepackage{lmodern}
\usepackage{amssymb}
\usepackage[colorlinks=true]{hyperref}
\usepackage{geometry}
\usepackage{tikz-cd}

\title[Expansion for a hypoelliptic diffusion]{Small time expansion\\for a strictly hypoelliptic kernel}
\author{Pierre Perruchaud}
\address{Départment de mathématiques, Université du Luxembourg, Luxembourg.}
\email{pierre.perruchaud@uni.lu}

\newtheorem{theorem}{Theorem}[section]
\newtheorem{lemma}[theorem]{Lemma}

\newtheorem{maintheorem}{Theorem}

\newtheorem{proposition}[theorem]{Proposition}
\newtheorem{corollary}[theorem]{Corollary}

\theoremstyle{definition}
\newtheorem{definition}[theorem]{Definition}
\newtheorem{remark}[theorem]{Remark}

\newcommand{\pconv}{\operatornamewithlimits{\ast}}

\begin{document}

\begin{abstract}
We consider the kernel of a hypoelliptic diffusion beyond the case of sub-ellipticity or polynomial coefficients. We get a full asymptotic expansion for small times, based on a Duhamel-type comparison with an approximate polynomial kernel. As in the sub-elliptic case, some change of scale based on the geometry of some Lie brackets yields a non-trivial limit for the kernel as time goes to zero. Remarkably, a different scale is needed to observe a non-trivial large deviation principle.
\end{abstract}

\maketitle

\contentsline {section}{\tocsection {}{1}{Introduction}}{1}{section.1}%
\contentsline {section}{\tocsection {}{2}{Strategy and notations}}{7}{section.2}%
\contentsline {section}{\tocsection {}{3}{Convolution calculus}}{10}{section.3}%
\contentsline {section}{\tocsection {}{4}{Study of the approximation}}{17}{section.4}%
\contentsline {section}{\tocsection {}{5}{Convergence of the series}}{22}{section.5}%
\contentsline {section}{\tocsection {}{6}{Fokker-Planck equation}}{24}{section.6}%
\contentsline {section}{\tocsection {}{7}{Small time expansions}}{29}{section.7}%
\contentsline {section}{\tocsection {}{8}{The macroscopic and mesoscopic scales}}{31}{section.8}\smallskip%
\contentsline {section}{\tocsection {Appendix}{A}{Qualitative regularity for the diffusion}}{36}{appendix.A}%
\contentsline {section}{\tocsection {Appendix}{B}{Closed form for the Fourier transform}}{38}{appendix.B}\smallskip%
\contentsline {section}{\tocsection {Appendix}{}{References}}{41}{section*.10}%

\section{Introduction}
\label{sec:introduction}

We consider the following stochastic differential equation over $\mathbb R^3$, for $W$ a one-dimensional standard Brownian motion:
\begin{align}
\mathrm dx_t &= \cos(\phi_t)\mathrm dt, &
\mathrm dy_t &= \sin(\phi_t)\mathrm dt, &
\mathrm d\phi_t &= \mathrm dW_t.
\end{align}
The initial condition is a deterministic $p_0=(x_0,y_0,\phi_0)\in\mathbb R^3$. Informally, $\phi$ is a Brownian angle, and $t\mapsto(x_t,y_t)$ is the planar curve with unit speed whose direction is given by $\phi$. We refer to $t\mapsto p_t=(x_t,y_t,\phi_t)$ as planar kinetic Brownian motion. We are interested in the geometry of the kernel of $t\mapsto p_t$ for small times $t>0$.
\medskip

The hypoellipticity criterion of Hörmander shows that $p_t$ admits a smooth density, denoted by $p\mapsto u_t(p_0,p)$, for all times $t>0$. Note however that the diffusion term on its own could not ensure existence of a density, since the motion would be confined to the lines of constant position $(x,y)$; we call this degeneracy condition \emph{strict hypoellipticity,} as opposed to \emph{sub-ellipticity.} The diffusion $t\mapsto p_t$ is of interest to us because it is one of the simplest, most natural example of a strictly hypoelliptic diffusion with non-polynomial coefficients.

Elliptic kernels (where the diffusion term has maximal rank) have been extensively studied, in particular because of their relationship with index theory. In a slightly less regular direction, the theory of sub-elliptic processes is now an active research area, with a plethora of structural results, going hand-to-hand with the development of sub-Riemannian geometry. In contrast, kernels of strictly non-elliptic diffusions are very poorly understood. As for sub-elliptic diffusions, the coefficients of the equation define a certain structure on the tangent space, called a filtration, which is expected to describe the directions which are increasingly difficult for the process to reach. More of this heuristics will be discussed in the following section. However, this relationship is currently still conjectural, and to the author's knowledge not even a natural subclass of examples has been proven to obey this rule of thumb.

A filtration of the tangent space is the data of increasing subspaces of the tangent space at $p_0$, for every point $p_0$ of the base space. We will discuss in section \ref{ssec:hypoelliptic-introduction} how a filtration arises from the data of a hypoelliptic diffusion, and find that the one associated to $t\mapsto p_t$ is given at $p_0=0$ by
\[      \mathbb R\partial_\phi
\subset \mathbb R\partial_\phi
\subset \mathbb R\partial_y\oplus\mathbb R\partial_\phi
\subset \mathbb R^3
\subset \mathbb R^3
\subset \cdots \]
For instance, we expect the direction $\partial_x$ (the last one to appear in the filtration) to be the least accessible by the diffusion. More generally, at a given $p_0$, the direction $\partial_y$ above has to be replaced by the space-like vector orthogonal to the velocity, i.e. $-\sin(\phi_0)\partial_x+\cos(\phi_0)\partial_y$. The aim of this paper is to describe in which sense this filtration indeed governs the small scale behaviour of the kernel for small times. We now state a qualitative corollary of our results that illustrates this principle.

\begin{maintheorem}
\label{mth:firstorder}
There exists a smooth function $\check u:\mathbb R^3\to\mathbb R$ such that
\[ u_t\big(0,(t+t^2\check x,t^{3/2}\check y,t^{1/2}\check\phi)\big)
 = t^{-4}\check u(\check p)+O(t^{-3}) \]
uniformly over $(t,\check p)\in[0,1]\times\mathbb R^3$. Moreover, $\check u$ is of Schwartz class, and
\[ \{\check p:\check u(\check p)>0\}=\{\check p:\check y<-2\check x\}. \]
\end{maintheorem}
In other words, around the solution $t\mapsto(t,0,0)$ of the noiseless system, the randomness of the process $t\mapsto p_t$ along the $n$-th step of the filtration has a typical scale of $t^n$. A closed form for the Fourier transform of $\check u$ is given in Appendix \ref{app:explicitFourier}.
\medskip

To state a stronger version of the above result, we must introduce the main protagonist of the proof, namely an approximation of the defining equation for kinetic Brownian motion, and the associated kernel $\widetilde u$. The approximation of the equation will be constructed as a second-order Taylor expansion of the coefficients. For instance, $p\mapsto\widetilde u_t(0,p)$ is the density of the solution $t\mapsto\widetilde p_t$ of the following system, evaluated at time $t$:
\begin{align*}
\mathrm d\widetilde x_t &= \big(1-\widetilde\phi_t^2/2\big)\mathrm dt, &
\mathrm d\widetilde y_t &= \widetilde\phi_t\mathrm dt, &
\mathrm d\widetilde\phi_t &= \mathrm dW_t.
\end{align*}
For a full description of $\widetilde u$ with general initial condition, see section \ref{ssec:deftildeu} below.

For all fixed initial condition $p_0\in\mathbb R^3$, $u$ must satisfy the Fokker-Planck equation
\[ \partial_tu = - \cos(\phi)\partial_xu - \sin(\phi)\partial_yu + \frac12\partial_\phi^2u  =: Lu, \]
where $L$ is the infinitesimal generator of $t\mapsto p_t$. This is not the case for $\widetilde u$, but since it is an approximation of $u$, we can hope that the error
\begin{equation} \label{eq:defE} E := -\partial_t\widetilde u+L\widetilde u \end{equation}
is small. The error $u-\widetilde u$, at least formally, is the space-time convolution $E*u$ (say by considering the derivative of the space convolution $s\mapsto\widetilde u_{t-s}\star u_s$), and in the spirit of the Duhamel principle, we should iteratively get
\[ u = \sum_{k\geq0}(E*)^k\widetilde u. \]
This is actually the main result of this work.

\begin{maintheorem}
\label{mth:Duhamel}
Define $\widetilde u$ and $E$ as in section \ref{ssec:deftildeu} and equation \eqref{eq:defE}. Then the kernel $u$ admits the representation
\begin{equation}
\label{eq:Duhamelseries}
u = \sum_{k\geq0}(E*)^k\widetilde u,
\end{equation}
where the sum on the right converges uniformly over $(t,p_0,p)\in(0,1)\times\mathbb R^3\times\mathbb R^3$, as well as all its derivatives.

Moreover, the $k$-th term of the above sum is bounded over the same set by a constant multiple of $t^{-4+k}$, and the function $\widetilde u$ is related to the function $\check u$ of Theorem \ref{mth:firstorder} through the relation
\[ \widetilde u_t\big(0,(t+t^2\check x,t^{3/2}\check y,t^{1/2}\check\phi)\big)
= t^{-4}\check u(\check p). \]
\end{maintheorem}

We will deduce Theorem \ref{mth:firstorder} above from this main result. The former also admits a higher order version.

\begin{maintheorem}
\label{mth:expansion}
There exists smooth functions $\check u_n:\mathbb R^3\to\mathbb R$ such that
\[ u_t\big(0,(t+t^2\check x,t^{3/2}\check y,t^{1/2}\check\phi)\big)
 = \sum_{n=0}^Nt^{-4+n/2}\check u_n(\check p)+O(t^{-7/2+N/2}) \]
uniformly over $(t,\check p)\in[0,1]\times\mathbb R^3$. Moreover, all the $\check u_k$ are of Schwartz class, $\check u_0$ is the function $\check u$ of Theorem \ref{mth:firstorder}, and $\check u_1=0$.
\end{maintheorem}

This shows that every higher order correction in the asymptotics also follows the behaviour prescribed by the filtration discussed above.
\medskip

In contrast, we show that for larger scales, the filtration does not seem to be the right framework to analyse the diffusion. Let us restrict the discussion to $p_0=0$. In the above we considered all functions as depending on the rescaled variables
\[ \Big(\frac{x-t}{t^2},\frac y{t^{3/2}},\frac\phi{t^{1/2}}\Big), \]
a scale we refer to as \emph{microscopic.} The unscaled variables correspond to the \emph{macroscopic} scale, and the intermediate \emph{mesoscopic} variables are given by
\[ \Big(\frac xt,\frac yt,\phi\Big). \]

In these larger scales, one expects the density to decay exponentially. A common framework for the description of exponential decay of densities is that of large deviation principles.

\begin{definition}
Let $I:\mathbb R^3\to[0,\infty]$ be lower semicontinuous. We say that $(u_t)_{t>0}$ satisfies a large deviation principle with rate function $I$ in the scale $\check p=\check p_t(p)$ if for all $U\subset F\subset\mathbb R^3$ with $U$ open and $F$ closed,
\[   -\inf_{\check p\in U}I(\check p)
\leq \liminf_{t\to0}t\log u_t(0,\check p_t^{-1}U) \\
\leq \limsup_{t\to0}t\log u_t(0,\check p_t^{-1}F)
\leq -\inf_{\check p\in F}I(\check p), \]
where we identify the density $p\mapsto u_t(0,p)$ with its corresponding probability measure.
\end{definition}

\begin{maintheorem}
\label{mth:LDP}
There exists a rate $I_\mathrm{macro}:\mathbb R^3\to[0,\infty]$ such that $u$ satisfies a large deviation principle with rate $I_\mathrm{macro}$ in the macroscopic scale. Moreover, this rate admits the explicit expression
\[ I_\mathrm{macro}:p = (x,y,\phi) \mapsto
   \begin{cases}
     \frac{|\phi|^2}2&\text{for }(x,y)=0, \\
     \infty&\text{else.}
   \end{cases} \]

There exists a rate $I_\mathrm{meso}:\mathbb R^3\to[0,\infty]$ such that $u$ satisfies a large deviation principle with rate $I_\mathrm{meso}$ in the mesoscopic scale. Moreover, $I_\mathrm{meso}(\check p)$ is finite if and only if $|(\check x,\check y)|<1$ or $\check p=(1,0,0)$.
\end{maintheorem}

This result is very closely related to the result \cite{Franchi-KBM} of Franchi, discussing an asymptotic development in the mesoscopic case. Our Theorem \ref{mth:LDP} is much more elementary than both Franchi's result and our above expansion. We discuss this relation in more details in section \ref{ssec:literature}.

We see that the macroscopic scale can only recover the part of the filtration induced by the diffusive part. By contrast, in the sub-elliptic case, the rate function corresponds to (one half of the square of) the Carnot-Carathéodory distance induced by the first level of the filtration, from which the higher levels can be extrapolated. In other words, the whole filtration is visible at the macroscopic level. In fact, in both cases, the macroscopic scale is completely unaware of the drift term, and is only sensible to noise; however, in the strictly hypoelliptic case, the drift becomes crucial to the study of the dynamics.

The mesoscopic scale remains mysterious. If $I_\mathrm{meso}$ being finite over an open set suggests that it is the right scale for a large deviation principle, there seem to be no obvious \emph{a priori} reason why $t$ would be the right factor. The rate function is not expected to be the square of a distance, although we do not prove it is not. It is also unclear that the rate function gives more information about the filtration.

\subsection{Hypoelliptic geometry}
\label{ssec:hypoelliptic-introduction}

Before we describe the filtration associated to $p$, we give some motivation from the elliptic and sub-elliptic cases. Let $p$ be a diffusion, say the solution of a stochastic differential equation of the form
\[ \mathrm dp_t = \sigma(p_t)\circ\mathrm dW_t + b(p_t)\mathrm dt \]
in $\mathbb R^d$, with smooth coefficients. Let us discuss first the most regular case of elliptic diffusions. This corresponds to the matrix $\sigma$ being surjective at every point, i.e. of rank $d$. In that case, some geometric information is easily extracted from $\sigma$, namely the Riemannian metric
\[ g:p_0\mapsto(\sigma(p_0)\sigma(p_0)^*)^{-1}. \]
Assuming $p$ is defined for all times, for instance if $\sigma$ and $b$ are uniformly bounded, then $p_t$ is well-defined, and it is known that
\[ \frac{p_t-p_0}{\sqrt t} \]
is asymptotically Gaussian for small $t$, with covariance $g(p_0)^{-1}$. If we rephrase this dynamical result in the language of kernels, it can be proved \cite[Chapter 3]{Kolokoltsov} that
\[      u_t\big(p_0,p_0+\sqrt t\check p\big)
 = t^{-d/2}\check u_{g(p_0)}(\check p)\big(1+O(\sqrt t)\big) + O\big(\exp(-C/t)\big), \]
where $C>0$ is a constant, the big $O$ notation stands for functions that can be bounded in absolute value by constant multiple of the argument (for $(t,p_0,\check p)$ in a compact set), and $\check u_g$ is the density of a centred Gaussian with variance $g^{-1}$. Explicitly,
\[ \check u(\check p)
 = \frac{\sqrt{\det g}}{(2\pi)^{d/2}}\exp\big(-\left|\check p\right|_g^2/2\big). \]
This relates the analytic behaviour of $u_t$ for small times to geometric information encoded in $\sigma$, namely the Riemannian metric $g$. It is worth noting that the information we get from $\sigma$ is actually independent of our representation of $p$, in the sense that any choice of $\tilde\sigma$ and $\tilde b$ inducing the dynamics of $p$ would yield the same metric $g$; one way of proving it is noting that the principal symbol of the generator of $p$ is precisely $g^{-1}$. It also turns out that in the higher order expansion of $u_t$ in powers of $\sqrt t$, the next term features the scalar curvature, thus providing another link between the small time behaviour of the density and the geometry induced by $p$.
\medskip

The situation becomes much richer when $\sigma$ does not have full rank. For instance, there has been much attention devoted to the family of sub-elliptic diffusions, which can be described (according to most authors) as follows. Define $\mathcal E$ as the smallest set of vector fields that contains all smooth fields $X$ with $X(p_0)\in\operatorname{im}\sigma(p_0)$ everywhere, and is closed under Lie bracket. For all $p_0$, the space $\mathcal E_{p_0}$ of all $X(p_0)$ for $X\in\mathcal E$ is a subspace of the tangent space at $p_0$. Then the diffusion $p$ is sub-elliptic if $\sigma(p_0)$ has constant rank and $\mathcal E_{p_0}$ is the whole space $\mathbb R^d$ at all points $p_0\in\mathbb R^d$. In that case, the collection of subspaces $\operatorname{im}\sigma$ is called a maximally non-integrable distribution, and it comes with a metric $g$ defined as above by $(\sigma\sigma^*)^{-1}$ (it is invertible over the distribution of hyperplanes). Going one step further, a sub-elliptic diffusion $p$ induces a natural filtration
\[ \mathcal E^1=\operatorname{im}\sigma\subset\mathcal E^2\subset\cdots\subset\mathcal E^\infty=\mathcal E \]
of the tangent spaces, given roughly by how many Lie brackets are necessary to generate a given vector.%
\footnote{The filtration collapses in finitely many steps at every given point, but not necessarily globally.}
Again, these objects depend on $\sigma$ and $b$ only through $p$, or more precisely its generator.

If the rank of these spaces $\mathcal E^n$ is constant (the so-called equiregular case), the rule of thumb is that the size of $p_t$ in the direction of $\mathcal E^n$ will be roughly $t^{n/2}$. Let us try to turn this intuition into a rigorous statement, through the lens of the density. We can decompose the tangent space at every point as the direct sum of some distributions $\mathcal F^n$, where $\mathcal F^1=\mathcal E^1$ and $\mathcal F^{n+1}$ is chosen (non-canonically) such that $\mathcal E^{n+1}=\mathcal E^n\oplus\mathcal F^{n+1}$. If we define the endomorphism $P_{\tau,p_0}$ of the tangent space at $p_0$ as the operator acting on $\mathcal F^n$ by multiplication by $\tau^n$, then one can show \cite[Theorem B]{CdVHT} that there exists a smooth non-zero function $\check u$ such that
\[ u_t\big(p_0,p_0+P_{\sqrt t,p_0}(\check p)\big) = t^{-D/2}\check u(p_0,\check p)+t^{-D/2}O(\sqrt t), \]
where $D$ is the sum of the weighted dimensions $n\cdot\operatorname{dim}\mathcal F^n$, so that $\det P_{\tau,p_0}=\tau^D$. In a more probabilistic language, the covariance matrix of $p_t-p_0$ is
\[ \mathbb E\big[(p_t-p_0)(p_t-p_0)^*\big]
\approx P_{\sqrt t,p_0}\Sigma P_{\sqrt t,p_0}^* \]
for some covariance matrix $\Sigma$, in the same sense as above. This shows that the filtration given by the $\mathcal E^n$ is encoded in the kernel, and a more careful analysis would link $g$ to the small time limit $\check u$.
\medskip

The diffusion we consider is not sub-elliptic: the curves $\{p:(x,y)=(x_0,y_0)\}$ are tangent to the image of $\sigma$, so the spaces $\mathcal E^n$ all have rank 1. However, we can ensure that $p_t$ admits a density under the weaker hypoelliptic hypothesis. Namely, define $\widetilde{\mathcal E}$ to be the smallest collection of vector fields containing $\mathcal E$ and stable under the adjoint action of $b$, i.e. under the action of the Lie bracket $[b,\cdot]$. The diffusion $p$ is said to be hypoelliptic if for all $p_0\in\mathbb R^d$, the subspace $\widetilde{\mathcal E}_{p_0}$ is actually the whole tangent space. As for sub-elliptic diffusions, general hypoelliptic ones also define a filtration of the tangent spaces in the following way. The initial collection $\widetilde{\mathcal E}^1$ of admissible vector fields are those with values in $\operatorname{im}\sigma$ (and by convention $\widetilde{\mathcal E}^0=\{0\}$). Inductively, $\widetilde{\mathcal E}^{n+2}$ consists of fields of one of three types: elements of $\widetilde{\mathcal E}^{n+1}$; brackets of elements of $\widetilde{\mathcal E}^{n+1}$; brackets of the form $[b,X]$, for $X\in\widetilde{\mathcal E}^n$. The union of all $\widetilde{\mathcal E}^n$ is the previously defined $\widetilde{\mathcal E}$, so the diffusion is hypoelliptic if and only if this construction stabilises (locally) to the whole tangent space.

In the case of our diffusion, by definition we have
\[ \widetilde{\mathcal E}^1 = \{f\partial_\phi,f:\mathbb R^3\to\mathbb R\}. \]
Brackets of vector fields of this type cannot generate any other field, so the next contribution to the filtration comes from the action of $b$ and will appear in the step $\widetilde{\mathcal E}^3$ of the filtration. We find
\[ [b,f\partial_\phi](0) = (\partial_xf)(0)\partial_\phi - f(0)\partial_y, \]
hence $\widetilde{\mathcal E}^3_0=\mathbb R\partial_y\oplus\mathbb R\partial_\phi$, as announced. Once can then check that indeed $\widetilde{\mathcal E}^4$ contains all smooth vector fields, which completes the filtration. As discussed above, we expect the variable $p_t$ to go as far as $t^{n/2}$ in the direction $\widetilde{\mathcal E}^n$, a fact indeed reflected in Theorem \ref{mth:firstorder}.
\medskip

In the same way that the next term in the expansion of $u$ is related to curvature in the elliptic case, we expect the second order term to be related to some notion of curvature for $T^1\mathbb R^2$. However, this curvature is constant, and so cannot be read from this first example. More insight should come from the study of kinetic Brownian motion on a surface, which reduces to that of compact perturbations of the Euclidean plane.

\subsection{State of the art}
\label{ssec:literature}

For sub-elliptic diffusions, and more generally sub-Riemannian geometry, the book \cite{ABB} is a very complete reference. Historically, the sub-elliptic equivalent of our large deviation result Theorem \ref{mth:LDP} is due to Léandre \cite{LeandreMin,LeandreMaj}, while the expansion of Theorem~\ref{mth:expansion} on the diagonal was given by Ben Arous in \cite{BenArous}. During the writing of this paper, the work \cite{CdVHT} was published on the arXiv, giving a local version of Theorem \ref{mth:expansion} for a wide range of sub-elliptic diffusions.

In terms of strictly hypoelliptic results, the author is aware of the following achievements.
\begin{itemize}
\item Several works of Franchi tackle the behaviour of the kernel of similar diffusions \cite{Franchi-Dudley,Franchi-Quad}. The unpublished work \cite{Franchi-KBM} considers the exact diffusion discussed here; in a sense, it gives a pointwise version of the large deviation principle of Theorem \ref{mth:LDP}. All these considerations offer a point of view somewhat intermediate between our Theorems \ref{mth:expansion} and \ref{mth:LDP}; namely, they apply to the value of the kernel (rather than a large deviation principle) but they take place in some mesoscopic scale (rather than some microscopic one). This means that Franchi's expansion depend on \emph{some} filtration even in the mesoscopic scale, which we cannot detect in our large deviation principle.

For instance, in the specific case of kinetic Brownian motion, the limit
\[ \lim_{t\to0}\log u_t\big(0,(t\check x,t\check y,\check\phi)\big) \]
(note the mesoscopic scale) is not a continuous function of $\check p$, but it is continuous over $\mathcal D^{i+1}\setminus\mathcal D^i$, for $\mathcal D$ the filtration
\[ 0\subset\mathbb R\partial_x\subset\mathbb R\partial_x\oplus\mathbb R\partial_y\subset\mathbb R^3\subset\cdots \]
The article of Franchi does not provide a geometric interpretation of this filtration. It does not seem that our methods could recover Franchi's results, nor do they readily shine light on this filtration.

\item In the case of less regular coefficients, the question of well-posedness and estimates on the density have been considered. For instance, in the article \cite{CdR}, Chaudru de Raynal, Honoré and Menozzi use a Duhamel principle to show existence and Schauder estimates for the kernel of a chain of oscillators (this covers for instance the case of the $n$th derivative being a standard Brownian motion).

\item The work of Kolokoltsov, for instance in the book \cite{Kolokoltsov}, describes a few cases for which a systematic approach can be followed. He discusses in particular a class of so-called \emph{regular} examples in Chapter 2, imposing amongst other conditions that the coefficients be quadratic, and gives in Chapter 3.6 an example where the same methods can be applied to a non-regular diffusion. However, the author was not able to apply these methods to the case at hand; see Chapter IV.2 of \cite{PThesis} for partial results.
\end{itemize}

\subsection*{Acknowledgments}
The author would like to thank the University of Rennes 1 and University of Notre Dame, where parts of this article were written.

I am indebted to Vassili Kolokoltsov, from the University of Warwick for many discussions around various questions involving the kernel of kinetic Brownian motion. His contribution ramifies far beyond the present work.

My gratitude also goes to Konstantin Izyurov, from the university of Helsinki, who provided the proof of Theorem \ref{th:checkuPsi} as an answer to a Mathoverflow question \cite{Izyurov}.

\section{Strategy and notations}
\label{sec:strategy}

\subsection{Approximate diffusion}
\label{ssec:deftildeu}

The general idea is that if two diffusions are solution to similar stochastic differential equations, their density should be similar as well, and in fact we will be able to expand one using the other as a building block. As much as possible, the quantities related to this approximation will be decorated with a tilde.

Recall that planar kinetic Brownian motion is the solution $t\mapsto p_t=(x_t,y_t,\phi_t)$ of the system of equations
\begin{equation} \label{eq:trueSDEcomponentwise} \left\{
\begin{aligned}
\mathrm dx_t & = \cos(\phi_t)\mathrm dt \\
\mathrm dy_t & = \sin(\phi_t)\mathrm dt \\
\mathrm d\phi_t & = \mathrm dW_t,
\end{aligned}
\right. \end{equation}
with initial condition $p_0$. This is the version over $\mathbb R^2$ of a known process: it has been given different names in the literature, such as kinetic Brownian motion or circular Langevin diffusion. In higher dimension, or on manifolds, the component $\phi$ represents a Brownian motion on the sphere, and $(x,y)$ is the movement of an object whose velocity is given by $\phi$.

For readability, it is convenient to introduce the vector fields $X$, $Y$ and $\Phi$ as follows:
\begin{equation}\label{eq:XYPhi}
    M_p := \begin{pmatrix} X(p) & Y(p) & \Phi(p) \end{pmatrix}
        := \begin{pmatrix}
             \cos(\phi) & - \sin(\phi) & 0 \\
             \sin(\phi) &   \cos(\phi) & 0 \\
             0 & 0 & 1
           \end{pmatrix},
\end{equation}
so $X$ is pointing in the direction indicated by $\phi$, $Y$ in the orthogonal direction, and $\Phi$ changes the angle. Their names come from their value at $0$: $X(0)=\partial_x$, $Y(0)=\partial_y$, $\Phi(0)=\partial_\phi$. Accordingly, we can rewrite equation \eqref{eq:trueSDEcomponentwise} as
\begin{equation} \label{eq:trueSDE}
\mathrm dp_t = X(p_t)\mathrm dt + \Phi(p_t)\mathrm dW_t,
\end{equation}
whose generator is given by
\[ L := X + \frac12\Phi^2. \]

The approximate equation will be given as a Taylor expansion about the initial condition $p_0$; as such, it is actually not a single approximate equation, but a family of them. Noting that
\[ X(p) = \cos(\phi-\phi_0)X(p_0) + \sin(\phi-\phi_0)Y(p_0), \]
we set $\widetilde\Phi := \Phi$ and
\[ \widetilde X(p) = \widetilde X^{(p_0)}(p) := \Big(1-\frac{(\phi-\phi_0)^2}2\Big)X(p_0) + (\phi-\phi_0)Y(p_0), \]
and define the approximation of equation \eqref{eq:trueSDE} at $p_0$ as
\begin{equation} \label{eq:tildeSDE}
\mathrm d\widetilde p_t = \widetilde X(\widetilde p_t)\mathrm dt + \widetilde\Phi(\widetilde p_t)\mathrm dW_t
\end{equation}
with initial condition $p_0$. In coordinates, this reads as
\begin{equation} \label{eq:tildeSDEcomponentwise} \left\{
\begin{aligned}
\mathrm d\widetilde x_t & = \left(1-\frac{(\widetilde \phi_t-\phi_0)^2}2\right)\cos(\phi_0)\mathrm dt
               - (\widetilde \phi_t-\phi_0)\sin(\widetilde \phi_t)\mathrm dt \\
\mathrm d\widetilde y_t & = \left(1-\frac{(\widetilde \phi_t-\phi_0)^2}2\right)\sin(\phi_0)\mathrm dt
               + (\widetilde \phi_t-\phi_0)\cos(\widetilde \phi_t)\mathrm dt \\
\mathrm d\widetilde \phi_t & = \mathrm dW_t.
\end{aligned}
\right. \end{equation}
The associated generator, for fixed $p_0$, is given by
\[ \widetilde L := \widetilde X + \frac12\widetilde\Phi^2. \]

\subsection{Duhamel principle}

Define the density $(t,p)\mapsto u_t(p_0,p)$ of the solution to equation \eqref{eq:trueSDE} started at $p_0$, and $(t,p)\mapsto\widetilde u_t(p_0,p)$ that of \eqref{eq:tildeSDE}. These functions are solution to the Fokker-Planck equations
\begin{align} \label{eq:FokkerPlanck}
&\left\{ \begin{aligned}
\partial_tu &= L^*u \\
u_0 &= \delta_{p_0}
\end{aligned} \right. &
&\text{and} &
&\left\{ \begin{aligned}\partial_t\widetilde u &=\widetilde L_{p_0}^*\widetilde u \\
\widetilde u_0 &= \delta_{p_0}.
\end{aligned} \right.
\end{align}
Over $\mathbb R^*_+\times\mathbb R^3$, these kernels are smooth since the equations are hypoelliptic; see Appendix \ref{app:Fellerandco}.

We expect that these densities are similar. To prove such a result, we will rely on the Duhamel principle. Letting $\mathsf e^{tL}$ be the semigroup associated with $L$, which acts on measures on the right, we have formally
\[ u_t-\widetilde u_t
 = \widetilde u_0\mathsf e^{tL} - \widetilde u_t\mathsf e^{0L}
 = \int_0^t\frac{\mathrm d}{\mathrm ds}\big(\widetilde u_{t-s}\mathsf e^{sL}\big)\mathrm ds
 = \int_0^t\big((-\partial_t+L^*)\widetilde u\big)_{t-s}\mathsf e^{sL}\mathrm ds. \]
Since $u$ is the fundamental solution to the Fokker-Planck equation associated with kinetic Brownian motion, we know that the action of the semigroup is no other than convolution with $u$. Moreover, $E:=(-\partial_t+L^*)\widetilde u$ is the error we get in the approximation: if it were zero, then $\widetilde u$ would be a solution of the Fokker-Planck equation for $u$. Writing $*$ (resp. $\star$) for the space-time (resp. space) convolution, we get
\[ u_t-\widetilde u_t
 = \int_0^tE_{t-s}\star u_s\mathrm ds
 = (E*u)_t. \]
This is a representation of $\widetilde u$ as a linear image of $u$: $\widetilde u = (\operatorname{id}-E*)u$. If we could invert this relation using the formal representation
\[ (\mathrm{id}-T)^{-1} = \sum_{k\geq0}T^k, \]
valid for, say, an operator $T$ of small norm in a Banach space, we could express $u$ as an expansion based on $\widetilde u$. It turns out that we can, and moreover we have reasonable understanding of the terms in these asymptotics. For instance, we will show that the $k$-th term is of order $O(t^{-4+k})$ as $t$ goes to zero. This is our main result; see Theorem \ref{mth:Duhamel} of the introduction for a statement.

\subsection{Scalings}
\label{ssec:scalings}

As discussed in the introduction, we expect the typical displacement in a given direction to be of order a power of $t$, whose exponent depends on the direction. Most of our analysis will take place in this scale; the corresponding dilated versions of all our objects will be decorated with a haček, with for instance $\check p$ representing a generic point in microscopic space.

The heuristics for the exact expression of the transform $\check p\mapsto p$ follow from the geometry of the iterated Lie brackets, as seen in section \ref{ssec:hypoelliptic-introduction}, but we can in this case give a simple intuition. The angle $\phi$ is a standard Brownian motion, so its size is roughly $t^{1/2}$. In the direction $Y(p_0)$, the variation is the integral of $\sin(\phi_t-\phi_0)\simeq t^{1/2}$, so we expect it to be of order $t^{3/2}$. Finally, in the direction $X(p_0)$, the tendency is to go straight ahead with velocity 1; however, the next term is the integral of $\cos(\phi_t-\phi_0)-1$, which means that the correction should be of order $t^2$.

The easiest way to guess, however, is probably to realise that the approximate diffusion can be explicitly solved. Namely,
\[ \widetilde p_t = p_0 + \Big(t-\frac12\int_0^tW_s^2\mathrm ds\Big)X(p_0) + \Big(\int_0^tW_s\mathrm ds\Big)Y(p_0) + W_t\Phi(p_0). \]
Recalling the definition of $M_p$ in equation \eqref{eq:XYPhi}, and introducing a rescaled Brownian motion $W^{(t)}:s\mapsto W_{st}/\sqrt t$, we find
\begin{equation}\label{eq:tildepintegral}
M_{p_0}^{-1}(\widetilde p_t-p_0)
 = \begin{pmatrix}
     t-t^2\frac12\int_0^1(W^{(t)}_s)^2\mathrm ds \\
     t^{3/2}\int_0^1W^{(t)}_s\mathrm ds \\
     t^{1/2}W^{(t)}_1
   \end{pmatrix}.
\end{equation}
Thus, for the case of the approximate diffusion, the scalings are, in a sense, exact.

The above reasons suggest to define the dilation operators
\begin{align}\label{eq:defTP}
T_\tau & := \begin{pmatrix}
              \tau^4 & 0 & 0 \\
              0 & \tau^3 & 0 \\
              0 & 0 & \tau
            \end{pmatrix}, &
P_{\tau,p_0}(\check p) := p_0 + \tau^2X(p_0) + M_{p_0}T_\tau\check p.
\end{align}
Let us set $\check u$ the density of
\begin{equation}
\label{eq:defxi}
\xi:= \begin{pmatrix}
        -\frac12\int_0^1W_s^2\mathrm ds \\
        \int_0^1W_s\mathrm ds \\
        W_1
      \end{pmatrix}.
\end{equation}
This is the same main contribution $\check u$ that features in Theorem \ref{mth:firstorder}. According to equation \eqref{eq:tildepintegral}, this is equivalent to defining
\begin{equation}\label{eq:defchecku}
\check u(\check p) := \tau^8u_{\tau^2}(p_0,P_{\tau,p_0}(\check p))
\end{equation}
for any fixed $\tau>0$ and $p_0\in\mathbb R^3$, where the factor $\tau^8$ corresponds to the Jacobian of $P_{\tau,p_0}$.

We will also use the following end-point version of $P$:
\begin{equation}
P_{-\tau,p}(\check p_0) := p - \tau^2X(p) + M_pT_\tau\check p_0.
\end{equation}
It is not quite an inverse of $P$ in the sense that we do not always have $\check p_0=-\check p$ for fixed $(\tau,p_0,p)$, but it is simpler and sufficient for our purposes.

\subsection{Function spaces}

We will define later a space $\check\Psi$ of well-behaved functions, see Definition~\ref{def:checkPsi}. It consists roughly of smooth functions of $(\tau,p_0,\check p)$ that are Schwartz in $\check p$, uniformly in $p_0$ and locally uniformly in $\tau$. For instance, we will see that the function $\check u$ defined above belongs to $\check\Psi$.

Define $\Psi^a$ as the space of functions $A:(t,p_0,p)\mapsto A_t(p_0,p)$ such that the rescaled version
\[ \check A:(\tau,p_0,\check p)\mapsto (\tau^2)^{5-a}A_{\tau^2}(p_0,P_{\tau^2,p_0}(\check p)) \]
belongs to $\check\Psi$. Since $\check u$ is in $\check\Psi$, then $\widetilde u$ must belong to $\Psi^1$ according to equation \eqref{eq:defchecku}. We will also see that the error $E$ is in $\Psi^1$ as well, see Proposition \ref{prop:EinPsi} and in fact it will be clear then that $\widetilde u$ has been defined exactly so that $E$ would satisfy this property.

The central piece of machinery is the fact that these spaces behave well under convolution: for $a,b>0$, we will see that $\Psi^a*\Psi^b\subset\Psi^{a+b}$. In the Duhamel formula, the $k$th convolution $(E*)^ku$ will then be an element of $\Psi^{k+1}$, behaving roughly as $O(t^{-4+k})$. For $k$ large, these become very regular, and a closer look at the convolutions shows that the series of equation \eqref{eq:Duhamelseries} actually converges in the smooth topology, in a rather strong sense --- see Theorem \ref{th:seriesconvergence}. Then we can prove that the limit satisfies the defining equation for $u$, so a uniqueness argument is enough to establish equation \eqref{eq:Duhamelseries}. The corollaries stated in the introduction follow, albeit not immediately.

\subsection{Outline}

We prove first the convergence of the Duhamel series. In section~\ref{sec:convolutions}, we introduce the spaces $\Psi^a$, and prove the fundamental convolution result in Theorem \ref{th:convolutionab}. We then prove in section~\ref{sec:checku} that the rescaled $\check u$ actually belongs to $\check\Psi$. As discussed above, these are the main ingredients in the proof of the convergence, which we establish in section \ref{sec:convergenceseries}.

In section \ref{sec:FokkerPlanck}, we show that the limit of the above series is solution to the defining equation for $u$ (Theorem \ref{th:FokkerPlanck}), with the right initial condition (Proposition \ref{prop:initialcondition}). A uniqueness argument then establishes equation \eqref{eq:Duhamelseries}. Section \ref{sec:corollaries} is dedicated to consequences of this representation, namely the derivation of small time expansions.

In the final section \ref{sec:LDP}, we discuss two large deviation results, as an illustration of the stark differences between our case and the better-behaved sub-elliptic setting.

\section{Convolution calculus}
\label{sec:convolutions}

\subsection{Function spaces}
We start by defining our spaces $\check\Psi$ and $\Psi^a$. For a given multi-index $\alpha=(\alpha_x,\alpha_y,\alpha_\phi)$ in $\mathbb N^3$ and smooth function $f:\mathbb R^3\to\mathbb R$ (or $\mathbb R^3\times E\to F$), set $D_p^\alpha f = \partial_x^{\alpha_x}\partial_y^{\alpha_y}\partial_\phi^{\alpha_\phi}f$.

\begin{definition}
\label{def:checkPsi}
We denote by $\check\Psi$ the set of smooth functions $\check A:[0;\infty)\times\mathbb R^3\times\mathbb R^3\to\mathbb R$ satisfying the following boundedness conditions. Denoting by $\tau$, $p_0$ and $\check p$ the variables of $\check A$,
\begin{equation}
\label{eq:defcheckpsi}
\sup_{0\leq\tau\leq T}\sup_{p_0,\check p}
   |\check p|^k\cdot\left|\partial_\tau^nD_{p_0}^\alpha D_{\check p}^\beta\check A(\tau,p_0,\check p)\right|
 < \infty
\end{equation}
for any time horizon $T>0$, any multi-indices $\alpha,\beta\in\mathbb N^3$, and any indices $n,k\in\mathbb N$.
\end{definition}

Note that they are smooth up to the boundary, so we can see them as restrictions of smooth functions defined on $\mathbb R\times\mathbb R^3\times\mathbb R^3$.

\medskip

The transformations $P_{\tau,p_0}$ and $P_{-\tau,p}$ are well-defined and smooth up to $\tau\geq0$. Moreover, they are invertible for positive times, a fact we use in the following definition.

\begin{definition}
\label{def:Psia}
For a fixed parameter $a\in\mathbb R$, a function $A:\mathbb R_+^*\times\mathbb R^3\times\mathbb R^3\to\mathbb R$ can be written as
\begin{equation*}
\stepcounter{equation}
\tag{\theequation~a}
\label{eq:defpsiasource}
A_t(p_0,p) = t^{-5+a}\check A^0\big(\sqrt t,p_0,P_{\sqrt t,p_0}^{-1}(p)\big),
\qquad\check A^0\in\check\Psi
\end{equation*}
if and only if it can be written as
\begin{equation*}
\tag{\theequation~b}
\label{eq:defpsiatarget}
A_t(p_0,p) = t^{-5+a}\check A^1\big(\sqrt t,p,P_{-\sqrt t,p}^{-1}(p_0)\big),
\qquad\check A^1\in\check\Psi.
\end{equation*}

We write $\Psi^a$ for the set of such functions.
\end{definition}

In other words, $A$ is in $\Psi^a$ if
\[ (\tau,p_0,\check p)\mapsto (\tau^2)^{5-a}A_{\tau^2}(p_0,P_{\tau,p_0}(\check p)) \]
extends to a smooth function on $\mathbb R\times\mathbb R^3\times\mathbb R^3$ rapidly decaying with respect to $\check p$; alternatively, if
\[ (\tau,p,\check p)\mapsto (\tau^2)^{5-a}A_{\tau^2}(P_{-\tau,p}(\check p),p) \]
extends to such a function.

\begin{proof}
Write $A$ as
\[ A_t(p_0,p) = t^{-5+a}\check A^0\big(\sqrt t,p_0,P_{\sqrt t,p_0}^{-1}(p)\big)
              = t^{-5+a}\check A^1\big(\sqrt t,p,P_{-\sqrt t,p}^{-1}(p_0)\big). \]
We must show that $\check A^0$ extends to a function in $\check\Psi$ if and only if $\check A^1$ does. But
\begin{align*}
\check A^0(\tau,p_0,\check p) & = \check A^1(\tau,P_{\tau,p_0}(\check p),\check p_0(\tau,p_0,\check p)),
\end{align*}
where
\begin{align*}
\check p_0(\tau,p_0,\check p)
& = P_{-\tau,\underbrace{P_{\tau,p_0}(\check p)}_{=:p}}^{-1}(p_0) \\
& = T_\tau^{-1}M_p^{-1}(p_0-p+\tau^2X(p)) \\
& = -(T_\tau^{-1}M_p^{-1}M_{p_0}T_\tau)\check p + \tau^2T_\tau^{-1}M_p^{-1}\big(X(p)-X(p_0)\big).
\end{align*}
All in all, we get
\[ \check A^0(\tau,p_0,\check p)
 = \check A^1\big(\tau,p_0+(*)_1\check p + (*)_2,(*)_3(-\check p + (*)_4)\big), \]
\begin{align*}
& \begin{aligned}
    (*)_1 & = M_{p_0}T_\tau, & (*)_2 & = \tau^2X(p_0), & (*)_3 & = T_\tau^{-1}M_p^{-1}M_{p_0}T_\tau,
  \end{aligned} \\
& (*)_4 = \tau^2T_\tau^{-1}M_{p_0}^{-1}\big(X(p)-X(p_0)\big).
\end{align*}
It turns out that $(*)_3$ and $(*)_4$ are smooth functions of $\tau$ and $(\phi-\phi_0)/\tau$, with all derivatives bounded by polynomials in those variables. This result will actually be used in various places, so we isolate the proof in the forthcoming Lemma \ref{lem:crucialquantities}. Since in our case this second variable is but $\check\phi$, the expression above makes it clear that derivatives of $\check A^0$ are bounded by polynomials in $\tau$, $\check p$ and the derivatives of $\check A^1$. In particular, if $\check A^1$ is in $\check\Psi$ then $\check A^0$ must be as well.

The reverse implication is similar; the equivalent representation is
\[ \check A^1(\tau,p,\check p_0)
 = \check A^0\big(\tau,p+(*)'_1\check p_0 - (*)'_2,(*)'_3(-\check p_0 - (*)'_4)\big), \]
\begin{align*}
& \begin{aligned}
   (*)'_1 & = M_pT_\tau, & (*)'_2 & = \tau^2X(p), & (*)'_3 & = T_\tau^{-1}M_{p_0}^{-1}M_pT_\tau,
  \end{aligned} \\
& (*)'_4 = \tau^2T_\tau^{-1}M_p^{-1}\big(X(p_0)-X(p)\big). \qedhere
\end{align*}
\end{proof}

\begin{lemma} \label{lem:crucialquantities}
Define the functions $\mathcal M$ and $\mathcal X$, respectively matrix- and vector-valued, by
\begin{align*}
\mathcal M(\tau,p_0,p) & = T_\tau^{-1}M_{p_0}^{-1}M_pT_\tau &
&\text{and} &
\mathcal X(\tau,p_0,p) & = \tau^2T_\tau^{-1}M_{p_0}^{-1}\big(X(p)-X(p_0)\big).
\end{align*}
Then there exists smooth functions $\check{\mathcal M}$ and $\check{\mathcal X}$ such that
\begin{align} \label{eq:checkMX}
  \mathcal M(\tau,p_0,p) & = \check{\mathcal M}(\tau,(\phi-\phi_0)/\tau) &
  &\text{and} &
  \mathcal X(\tau,p_0,p) & = \check{\mathcal X}(\tau,(\phi-\phi_0)/\tau).
\end{align}
Moreover, $\check{\mathcal M}$ and $\check{\mathcal X}$, along with all their derivatives, are bounded by polynomials in their variables $\tau$ and $\check\phi$.
\end{lemma}

\begin{proof}
Since $M_{p_0}$ is the matrix of a rotation of angle $\phi_0$ in the $(x,y)$-plane, it is clear that $M_p=M_{p_0}M_{p-p_0}$ and $X(p) = M_p(1,0,0)$. Accordingly, the above matrix and vector field are
\begin{align*}
& T_\tau^{-1}M_{p-p_0}T_\tau
& \text{and} &
& \tau^2T_\tau^{-1}(M_{p-p_0}-\mathrm{Id})\begin{pmatrix} 1 \\ 0 \\ 0 \end{pmatrix}.
\end{align*}
Using the actual expressions for $T_\tau$ and $M_{p-p_0}$, and setting $\check\phi$ such that $\phi = \phi_0+\tau\check\phi$, we find explicitly
\begin{align*}
& \begin{pmatrix}
    \cos(\tau\check\phi) & - \frac1\tau\sin(\tau\check\phi) & 0 \\
    \tau\sin(\tau\check\phi) &   \cos(\tau\check\phi) & 0 \\
    0 & 0 & 1
  \end{pmatrix}
& \text{and} &
& \begin{pmatrix} \frac1{\tau^2}(\cos(\tau\check\phi)-1) \\ \frac1\tau\sin(\tau\check\phi) \\ 0 \end{pmatrix}.
\end{align*}
Clearly these depend only on $\tau$ and $\check\phi$. Setting $\operatorname{sinc}$ the classical $\phi\mapsto\sin(\phi)/\phi$ function, we get
\begin{align*}
   & \begin{pmatrix}
       \cos(\tau\check\phi) & - \operatorname{sinc}(\tau\check\phi)\check\phi & 0 \\
       \tau\sin(\tau\check\phi) &   \cos(\tau\check\phi) & 0 \\
       0 & 0 & 1
     \end{pmatrix}
   & \text{and} &
   & \begin{pmatrix} -\frac12\operatorname{sinc}(\tau\check\phi/2)^2\check\phi^2 \\ \operatorname{sinc}(\tau\check\phi)\check\phi \\ 0 \end{pmatrix}.
\end{align*}
Now it suffices to control the derivatives of $\operatorname{sinc}$. This follows from an easy induction argument, namely one can show that the derivatives $\phi^{n+1}\operatorname{sinc}^{(n)}(\phi)$ stay of the form
\[ P(\phi)\cos(\phi)+Q(\phi)\sin(\phi) \]
for well-chosen polynomials $P$ and $Q$ and of degree at most $n$. This suffices to get $\mathcal M$ and $\mathcal X$, along with their derivatives, uniformly bounded by polynomials in $\tau$ and $\check\phi$.
\end{proof}

\subsection{Convolution theorems}
We are now ready to state and prove the main result of this section.

\begin{theorem}
\label{th:convolutionab}
For any $A\in\Psi^a$ and $B\in\Psi^b$, if $a>0$ and $b>0$ the convolution
\[ A*B:(t,p_0,p)\mapsto\int_0^t\int_{\mathbb R^3}A_s(p_0,p_*)B_{t-s}(p_*,p)\mathrm d s\,\mathrm d p_* \]
is well-defined and belongs to $\Psi^{a+b}$. In symbols, $\Psi^a*\Psi^b\subset\Psi^{a+b}$.
\end{theorem}

The following lemma will appear as an intermediate result along the proof of Theorem \ref{th:convolutionab}, and sums up the essence of the reasoning, in some uniform sense rather than the Schwartz one.

\begin{lemma}
\label{lem:new-uniformab}
For $A\in\Psi^a,B\in\Psi^b$ and $0<\varepsilon\leq1-\delta<1$, define the partial convolution
\[ \big(A\pconv_\varepsilon^{1-\delta} B\big):(t,p_0,p)
   \mapsto
   t\int_\varepsilon^{1-\delta}\int_{\mathbb R^3}A_{rt}(p_0,p_*)B_{(1-r)t}(p_*,p)\mathrm dp_*\mathrm dr. \]

The integrand is locally uniformly integrable; in particular the partial convolution is continuous. 

Moreover, if $a>0$, then the convergence
\[ A\pconv_\varepsilon^{1/2}B \to A\pconv_0^{1/2}B\]
holds locally uniformly in $(t,p_0,p)$ (and the right hand side is well-defined). The similar result holds for $b>0$.
\end{lemma}

In the local uniform convergence above, only the cases $t$ small or large are actually troublesome. The convergence holds uniformly in $(p_0,p)$, but we will not need this result.

\begin{proof}[Proof of Theorem \ref{th:convolutionab}]
In the following, the integrals under consideration are possibly undefined for now, but we show along the proof that they are equal to well-defined quantities. This means that a formal proof would go from the bottom up.

Since $A$ and $B$ are in $\Psi^a$ and $\Psi^b$, by definition there must exists $\check A,\check B\in\check\Psi$ such that $A$ and $B$ satisfy
\begin{align*}
  \tau^{2(5-a)}A_{\tau^2}(p_0,P_{\tau,p_0}(\check p)) & = \check A(\tau,p_0,\check p), \\
  \tau^{2(5-b)}B_{\tau^2}(P_{-\tau,p}(\check p_0),p) & = \check B(\tau,p,\check p_0).
\end{align*}

Consider first the time integral from $0$ to $t/2$, which we call $I^A$. We will show that a proper renormalisation $\check I^A$ belongs to $\check\Psi$; specifically,
\begin{align*}
\check I^A:(\tau,p_0,\check p)
  & \mapsto \tau^{2(5-a-b)}I^A(\tau^2,p_0,P_{\tau,p_0}(\check p)) \\
  & \qquad = \tau^{2(5-a-b)}\int_0^{\tau^2/2}\int_{\mathbb R^3}
                            A_s(p_0,p_*)B_{\tau^2-s}(p_*,P_{\tau,p_0}(\check p))\mathrm d s\,\mathrm d p_*.
\end{align*}
We change variables in the defining integral for $\check I^A$, namely we set $s=(\sigma\tau)^2$ and $p_* = P_{\sigma\tau,p_0}(\check p_*)$. This gives
\[ \check I^A(\tau,p_0,\check p)
   = 2\int_0^{1/\sqrt2}\int_{\mathbb R^3}
      \sigma^{-1+2a}\theta^{-2(5-b)}\check A(\sigma\tau,p_0,\check p_*)\check B(\theta\tau,p,\check p_B)\mathrm d\sigma\,\mathrm d \check p_*, \]
\begin{align*}
    \theta & = \theta_\sigma = \sqrt{1-\sigma^2}, \\
         p & = p(\tau,p_0,\check p) = P_{\tau,p_0}(\check p), \\
\check p_B & = (\check p_B)_{\sigma,\check p_*}(\tau,p_0,\check p) = P_{-\theta\tau,P_{\tau,p_0}(\check p)}^{-1}\big(P_{\sigma\tau,p_0}(\check p_*)\big).
\end{align*}
Note that the hypothesis $a>0$ becomes clear, for the convergence of the integral.

Similarly, we can consider the integral $I^B$ from $t/2$ to $t$, and its endpoint normalisation
\begin{align*}
\check I^B:(\tau,p,\check p_0)
  & \mapsto \tau^{2(5-a-b)}\int_{\tau^2/2}^{\tau^2}\int_{\mathbb R^3}
                             A_s(P_{-\tau,p}(\check p_0),p_*)B_{\tau^2-s}(p_*,p)\mathrm d s\,\mathrm d p_*.
\end{align*}
The change of variables $t = \tau^2(1-\sigma^2)$, $p_*=P_{-\sigma\tau,p}(\check p_*)$ gives
\[ \check I^B(\tau,p,\check p_0)
   = 2\int_0^{1/\sqrt2}\int_{\mathbb R^3}
      \sigma^{-1+2b}\theta^{-2(5-a)}\check A(\theta\tau,p_0,\check p_A)\check B(\sigma\tau,p,\check p_*)\mathrm d\sigma\,\mathrm d \check p_*, \]
\begin{align*}
    \theta & = \theta_\sigma = \sqrt{1-\sigma^2}, \\
       p_0 & = p_0(\tau,p,\check p_0) = P_{-\tau,p}(\check p_0), \\
\check p_A & = (\check p_A)_{\sigma,\check p_*}(\tau,p,\check p_0) = P_{\theta\tau,P_{-\tau,p}(\check p_0)}^{-1}\big(P_{-\sigma\tau,p}(\check p_*)\big).
\end{align*}

It will be enough to show that both $\check I^A$ and $\check I^B$ belong to $\check\Psi$. We will be able to take the derivative under the integral sign and conclude, provided we can show the estimates
\begin{gather*}
\stepcounter{equation}
\tag{\theequation~a}
\label{eq:uniformbound0}
  \sup_{\substack{\sigma<1-\varepsilon\\\tau\leq T}}\sup_{p_0,\check p,\check p_*}
  |\check p|^k|\check p_*|^\ell
  \cdot\left|\partial_\tau^nD_{p_0}^\alpha D_{\check p}^\beta\big(\theta^{-2(5-b)}\check A(\sigma\tau,p_0,\check p_*)\check B(\theta\tau,p,\check p_B)\big)\right| < \infty \\
\tag{\theequation~b}
\label{eq:uniformbound1}
  \sup_{\substack{\sigma<1-\varepsilon\\\tau\leq T}}\sup_{p,\check p_0,\check p_*}
  |\check p_0|^k|\check p_*|^\ell
  \cdot\left|\partial_\tau^nD_p^\alpha D_{\check p_0}^\beta\big(\theta^{-2(5-b)}\check A(\theta\tau,p_0,\check p_A)\check B(\sigma\tau,p,\check p_*)\big)\right| < \infty
\end{gather*}
for all $\varepsilon>0$, $k,\ell\in\mathbb N$, and $\alpha,\beta\in\mathbb N^3$. Note that they would imply the estimates of Lemma \ref{lem:new-uniformab}, by operating on the considered integrals the same change of variables.

Since $\check A$ and $\check B$ are in $\check\Psi$, and can in a certain sense absorb powers of $\check p_*$, $\check p_A$ and $\check p_B$, the proof of the estimates above reduces to the following two facts.
\begin{enumerate}
\item The quantities $\theta$, $p$ and $\check p_B$ are smooth with respect to $(\tau,p_0,\check p)$, with all derivatives bounded by polynomials in $(\tau,\check p,\check p_*)$, uniformly in $\sigma$ and $p_0$. Similarly, the quantities $\theta$, $p_0$ and $\check p_A$ are smooth with respect to $(\tau,p,\check p_0)$, with all derivatives bounded by polynomials in $(\tau,\check p,\check p_*)$, uniformly in $\sigma$ and $p$.
\item We have $|\check p|\lesssim (1+\tau+|\check p_*| + |\check p_B|)^2$, up to a universal constant. Similarly, $|\check p_0|\lesssim (1+\tau+|\check p_*| + |\check p_A|)^2$
\end{enumerate}

Let us show first that $\theta$, $p$ and $\check p_B$ above, along with all their derivatives, are bounded by polynomials in $(\tau,\check p,\check p_*)$. The result is obvious for $\theta$. Regarding $p$ and $\check p_B$, direct computations yield
\begin{align*}
         p & = p_0 + M_{p_0}T_\tau\check p + \tau^2X(p_0), \\
\check p_B & = - \mathcal X(\theta\tau,p,p_0) + \mathcal M(\theta\tau,p,p_0)T_{\sigma/\theta}\check p_* - T_{1/\theta}\mathcal M(\tau,p,p_0)\check p \\
           & = - \check{\mathcal X}(\theta\tau,-\check\phi/\theta) + \check{\mathcal M}(\theta\tau,-\check\phi/\theta)T_{\sigma/\theta}\check p_* - T_{1/\theta}\check{\mathcal M}(\tau,-\check\phi)\check p,
\end{align*}
with $\mathcal M$ and $\mathcal X$ defined as in Lemma \ref{lem:crucialquantities}. According to the lemma, and since $\theta$ is bounded below, these quantities, along with all their derivatives, are bounded by polynomials in $(\tau,\check p,\check p_*)$, uniformly in $\sigma$ and $p_0$.

The quantities pertaining to $\check I^B$ are controlled in the same fashion, using the expressions
\begin{align*}
       p_0 & = p + M_pT_\tau\check p_0 - \tau^2X(p), \\
\check p_A & = \mathcal X(\theta\tau,p_0,p) + \mathcal M(\theta\tau,p_0,p)T_{\sigma/\theta}\check p_* - T_{1/\theta}\mathcal M(\tau,p_0,p)\check p_0 \\
           & = \check{\mathcal X}(\theta\tau,-\check\phi_0/\theta) + \check{\mathcal M}(\theta\tau,-\check\phi_0/\theta)T_{\sigma/\theta}\check p_* - T_{1/\theta}\mathcal M(\tau,-\check\phi_0)\check p_0.
\end{align*}

As for the second fact, we want to bound $|\check p|$ by a polynomial in $|\check p_*|$ and $|\check p_B|$. In essence, this will come from the fact that $p-p_0$ is but $(p-p_*)+(p_*-p_0)$. Specifically, by definition of $p$ and $\check p_B$,
\[ P_{\tau,p_0}(\check p)-p_0 = (p - P_{-\theta\tau,p}(\check p_B)) + (P_{\sigma\tau,p_0}(\check p_*)-p_0). \]
Unfolding the definitions,
\[ \tau^2X(p_0) + M_{p_0}T_\tau\check p
 = \big((\theta\tau)^2X(p) - M_pT_{\theta\tau}\check p_B\big)
 + \big((\sigma\tau)^2X(p_0) - M_{p_0}T_{\sigma\tau}\check p_*\big), \]
so that we have
\begin{align*}
\check p & = \theta^2\mathcal X(\tau,p_0,p) + T_\sigma\check p_* - \mathcal M(\tau,p_0,p)T_\theta\check p_B \\
         & = \theta^2\check{\mathcal X}(\tau,\check\phi) + T_\sigma\check p_* - \check{\mathcal M}(\tau,\check\phi)T_\theta\check p_B.
 \end{align*}
Actually, at this point we need a bit more about $\mathcal M$ and $\mathcal X$ than can be found in Lemma \ref{lem:crucialquantities}. That said, using the exact expressions in equation \eqref{eq:checkMX}, we get that
\begin{align*}
|\check\phi| & = |\sigma\check\phi_*- \theta\check\phi_B| \leq |\check p_*| + |\check p_B|, \\
  |\check y| & \leq \theta^2|\check\phi| + \sigma^3|\check y_*| + |\theta^4\tau\check x_B| + |\theta^3\check y_B|
               \leq 2|\check p_*| + (2+\tau)|\check p_B|, \\
  |\check x| & \leq \theta^2|\check\phi^2| + \sigma^4|\check x_*| + |\theta^4\check x_B| + |\check\phi\theta^3\check y_B|
               \leq 2(1+|\check p_*|+|\check p_B|)^2.
\end{align*}
This is the required inequality for $|\check p|$ as it appears in $\check I^A$. The equivalent inequality for $\check I^B$ is deduced in the same fashion from
\[ p-P_{-\tau,p}(\check p_0) = (p - P_{-\sigma\tau,p}(\check p_*)) + (P_{\theta\tau,p_0}(\check p_A) - p_0), \]
leading to
\begin{align*}
\check p_0 & = -\theta^2\check{\mathcal X}(\tau,\check\phi_0) + T_\sigma\check p_* - \check{\mathcal M}(\tau,\check\phi_0)T_\theta\check p_A.
\end{align*}
This concludes the proof of the second fact, hence that of Theorem \ref{th:convolutionab}.
\end{proof}

Provided, as we will see in Section \ref{sec:checku}, that $\widetilde u$ and $E$ are elements of $\Psi^1$, this convolution theorem shows that the $k$th term $(E*)^k\tilde u$ in the expansion \eqref{eq:Duhamelseries} belongs to the space $\Psi^{k+1}$. To show that the series converges, we need some quantitative uniform control on top of this qualitative result. These estimates will be easier to establish when the functions under consideration extend to the boundary $t=0$, which happens for the elements of $\Psi^5$. Accordingly, we would like to reduce the study of iterated convolutions on the left to that of convolution of smooth functions arising as instances of $E^{*5}\in\Psi^5$; of course, we would like the convolution operation to be associative whenever well-defined. Thus, the remainder of this section is devoted first to the associativity of $*$, as described in Proposition \ref{prop:nconvolution} using the method of proof of the convolution Theorem \ref{th:convolutionab}; and secondly to the Proposition \ref{prop:regularconvolution}, giving a quantitative control for the convolution of regular functions.

\begin{proposition} \label{prop:nconvolution}
Let $A_0,\cdots,A_\ell$ be elements of $\Psi^{a_0},\cdots,\Psi^{a_\ell}$ respectively, with each $a_i$ positive. Then, for any choice of parentheses for the expression $A_0*\cdots*A_\ell$, the convolution is well-defined, belongs to $\Psi^{a_0+\cdots+a_\ell}$, and its value at $(t,p_0,p)$ is given by the absolutely convergent integral
\[ \int_{0\leq t_1\leq\cdots\leq t_\ell\leq t}\int_{\mathbb R^{3\ell}}
     A_0(t_1,p_0,p_*^1)\cdots A_i(t_{i+1}-t_i,p_*^i,p_*^{i+1}) \cdots A_\ell(t-t_\ell,p_*^\ell,p)
     \mathrm dt_1\cdots\mathrm dt_\ell\,\mathrm dp_*. \]
In particular, we can write $A_0*\cdots*A_\ell$ unambiguously, and $*$ is associative for functions in $\bigcup_{a>0}\Psi^a$.
\end{proposition}

\begin{proof}
Theorem \ref{th:convolutionab} states exactly that the convolution is a well-defined element of $\Psi^{a_1+\cdots+a_\ell}$ for any given choice of parentheses. It is clear that if the integrand is in $\mathrm L^1$, then we can swap integrals around and show that it equals an expression with parentheses of our liking. On the other hand, such an $\mathrm L^1$ estimate would follow if we were to show that some choice of parentheses makes $|A_0|*\cdots*|A_\ell|$ well-defined.

This is an easy consequence of the proof of Theorem \ref{th:convolutionab}. Namely, write $\check\Psi_0$ for the version of $\check\Psi$ where we forget about derivatives (i.e. $n=0$ and $\alpha=\beta=0$ in equation \eqref{eq:defcheckpsi}), and $\Psi^a_0$ for the corresponding scaled space (see Definition \ref{def:Psia}). If $|A|$ and $|B|$ are in $\Psi^a_0$ and $\Psi^b_0$ for $a,b>0$, then the estimates \eqref{eq:uniformbound0} and \eqref{eq:uniformbound1} show that $|A|*|B|$ is in $\Psi^{a+b}_0$.
\end{proof}

We introduce the operator $\mathcal T:\Psi^a\to\Psi^{a+1}$ characterised by
\begin{equation}
\label{eq:defT}
\mathcal Tv:(t,X_0,X_1)\mapsto tv(t,X_0,X_1).
\end{equation}
Of course, there is an inverse $\mathcal T^{-1}:\Psi^a\to\Psi^{a-1}$ with obvious definition. Note also that $\mathcal T(A*B) = (\mathcal TA)*B + A*(\mathcal TB)$ whenever each term is well-defined.

\begin{proposition}
\label{prop:regularconvolution}
Let $\mathcal A$ be a finite subset of $\Psi^5$ and $T>0$ some time horizon. There exists a constant $M>0$ such that for any tuple $(A_0,\ldots,A_\ell)$ of functions in $\mathcal A$ and any $m\leq5\ell$,
\begin{equation}
\label{eq:supnormnconvolution}
|\mathcal T^{-m}(A_0*\cdots*A_\ell)|_\infty \leq \frac{M^{\ell+1}\cdot T^{5\ell-m}}{\ell!}
\end{equation}
for $|A|_\infty$ the supremum of $A$ over $(0,T]\times\mathbb R^3\times\mathbb R^3$.
\end{proposition}

\begin{proof}
For any given $A\in\Psi^5$, by definition there exists a constant $M$ such that
\[ \big|A_{\tau^2}\big(p_0,P_{\tau,p_0}(\check p)\big)\big| \leq \frac M{1+|\check p|^4} \]
when $\tau^2\leq T$, uniformly in $(p_0,\check p)$. Since $\mathcal A$ is finite, we can choose $M$ such that the inequality holds for every one of its elements.

Using the representation given in Proposition \ref{prop:nconvolution}, we get
\begin{multline*}
|(A_0*\cdots*A_\ell)(t,p_0,p)| \\
\begin{aligned}
& \leq \iint \frac M{1+\Big|P_{\sqrt{t_1},p_0}^{-1}(p_*^1)\Big|^4}
             \cdots\frac M{1+\Big|P_{\sqrt{t_{i+1}-t_i},p_*^i}^{-1}(p_*^{i+1})\Big|^4}\cdots
             \frac M{1+\Big|P_{\sqrt{t-t_\ell},p_*^\ell}^{-1}(p)\Big|^4}
             \mathrm dt\,\mathrm dp_* \\
& = M^{\ell+1}
    \iint \frac{t_1^4}{1+|\check p_*^1|^4}
          \cdots\frac{(t_{i+1}-t_i)^4}{1+|\check p_*^{i+1}|^4}\cdots
          \frac{(t_\ell-t_{\ell-1})^4}{1+|\check p_*^\ell|^4}\frac1{1+0}
          \mathrm dt\,\mathrm d\check p_*,
\end{aligned}
\end{multline*}
where the integrals range over $0\leq t_1\leq\cdots\leq t_\ell\leq t$, $p_*,\check p_*\in(\mathbb R^3)^\ell$. This gives
\begin{align*}
|(A_0*\cdots*A_\ell)(t,p_0,p)|
& \leq M^{\ell+1}
       \Big(\int_{t_1\leq\cdots\leq t_\ell\leq t}
            t_1^4\cdots(t_\ell-t_{\ell-1})^4
            \mathrm dt\Big)
       \Big(\int_{\mathbb R^3}\frac{\mathrm d\check p}{1+|\check p|^4}\Big)^\ell \\
& = M^{\ell+1}t^{5\ell}
    \Big(\int_{t_1\leq\cdots\leq t_\ell\leq 1}
         t_1^4\cdots(t_\ell-t_{\ell-1})^4
         \mathrm dt\Big)
    \big(\sqrt2\pi^2\big)^\ell.
\end{align*}
A crude estimate gives
\[   \int_{t_1\leq\cdots\leq t_\ell\leq 1}t_1^4\cdots(t_\ell-t_{\ell-1})^4\mathrm dt
\leq \int_{t_1\leq\cdots\leq t_\ell\leq 1}1\mathrm dt
   = \frac1{\ell!}\int_{t_1,\cdots,t_\ell\leq 1}1\mathrm dt = \frac1{\ell!}, \]
from which the proposition follows.
\end{proof}

\subsection{Derivatives}
Clearly, functions in $\Psi^a$ are smooth since their normalised versions are. Moreover, their derivatives actually belong to the same hierarchy of $\Psi$, albeit with a worse exponent. For instance, taking the derivative of equation \eqref{eq:defpsiasource} with respect to $x$, we see that
\begin{align*}
        \partial_xA_t(p_0,p)
     &= t^{-5+a-2}\cos(-\phi_0)\partial_{\check x}\check A\big(\sqrt t,p_0,P_{\sqrt t,p_0}^{-1}(p)\big) \\
&\qquad + t^{-5+a-3/2}\sin(-\phi_0)\partial_{\check y}\check A\big(\sqrt t,p_0,P_{\sqrt t,p_0}^{-1}(p)\big),
\end{align*}
so $\partial_xA\in\Psi^{a-2}$. In fact, we have the following result by induction.

\begin{proposition}
\label{prop:lossofregularity}
Let $A$ be in $\Psi^a$. For any $\alpha_0,\alpha\in\mathbb N^3$, $n\in\mathbb N$,
\[ \partial_t^nD_{p_0}^{\alpha_0}D_p^{\alpha}A\in\Psi^{a-2K}, \]
for $K = n + (\alpha_0^x+\alpha_0^y+\alpha_0^\phi) + (\alpha^x+\alpha^y+\alpha^\phi)$.
\end{proposition}
Actually, derivatives in the direction of $\phi$ and $\phi_0$ would only worsen the exponent by $1/2$ rather than 2, but it will not help us in the following.

However, sometimes these derivatives are not easily computed. For instance, if we know that $A$ and $B$ are functions of $\Psi^a$ for some $a>0$, the relative technicality of proof of the existence and smoothness of $A*B$, see Theorem \ref{th:convolutionab}, suggests that its derivatives have no reason to admit a clean expression. That said, if $a>0$ is large enough, we can show the following. Recall that $\mathcal T$ is the multiplication by $t$ operator.

\begin{proposition}
\label{prop:derivativesconvolution}
Let $A\in\Psi^a$ and $B\in\Psi^b$ be some functions of regularity $a,b>0$, and $\alpha=(\alpha_x,\alpha_y,\alpha_\phi)\in\mathbb N^3$. We set $K=2(\alpha_x+\alpha_y+\alpha_\phi)$. Then, over $(0,\infty)\times\mathbb R^3\times\mathbb R^3$, the derivatives of $A*B$ are given by
\[ D_{p_0}^\alpha(A*B)
 = (D_{p_0}^\alpha A)*B\text{ whenever }a>K, \]
\[ D_p^\alpha(A*B)
 = A*(D_p^\alpha B)\text{ whenever }b>K, \]
\[ \partial_t(A*B)
 = \mathcal T^{-1}\big(A*B + (\mathcal T\partial_tA)*B + A*(\mathcal T\partial_tB)\big)\text{ whenever }a,b>1. \]
\end{proposition}
For the sake of conciseness, we will use that last formula in the form
\[ (\partial_t\mathcal T)(A*B)
 = \big((\partial_t\mathcal T)A\big)*B + A*\big((\partial_t\mathcal T)B\big) \]
for $a,b>1$, which we deduce from the obvious Heisenberg commutator $\partial_t\mathcal T-\mathcal T\partial_t = \mathrm{id}$.

\begin{proof}
We describe the situation with the single derivative $\partial_{x_0}$, i.e. we prove
\[ \partial_{x_0}\big(A*B\big)
 = \big(\partial_{x_0}A\big)*B \]
for $a>2$, $b>0$. The other derivatives behave in the same way, and the result follows by induction.

Because the integrand of the partial convolution $(\partial_{x_0}A)\pconv_\varepsilon^{1-\varepsilon}B$ is locally uniformly integrable, the derivative commutes with the partial convolution, and we have
\[ \partial_{x_0}\Big(A\pconv_\varepsilon^{1-\varepsilon}B\Big)
= \big(\partial_{x_0}A\big)\pconv_\varepsilon^{1-\varepsilon}B.\]
Suppose $a>2$ and $b>0$. According to Lemma \ref{lem:new-uniformab}, we know that the following sequence of equalities hold, in the sense of distributions.
\begin{align*}
   \partial_{x_0}\big(A*B\big)
&= \partial_{x_0}\Big(\lim_{\varepsilon\to0}A\pconv_\varepsilon^{1-\varepsilon}B\Big)
 = \lim_{\varepsilon\to0}\partial_{x_0}\Big(A\pconv_\varepsilon^{1-\varepsilon}B\Big) \\
&= \lim_{\varepsilon\to0}\Big(\big(\partial_{x_0}A\big)\pconv_\varepsilon^{1-\varepsilon}B\Big)
 = \big(\partial_{x_0}A\big)*B.
\end{align*}
This shows that the expected equality of functions holds in the sense of distributions; however, since both sides are actually smooth functions (from Theorem \ref{th:convolutionab}), they coincide everywhere as announced.
\end{proof}

\section{Study of the approximation}
\label{sec:checku}

Recall that $p\mapsto\widetilde u_t(p_0,p)$ is the density of $p_t$, where $p$ is the solution to equation \eqref{eq:tildeSDE} with initial condition $p_0$. As described in Section \ref{ssec:scalings}, in the coordinates given by the transformations $P_{\tau,p_0}$, $\widetilde u$ is independent of time and the initial condition, see in particular equation \eqref{eq:defchecku} where the rescaled version $\check u$ is defined.

Recall that $\check u$ is the density of the vector
\[ \xi
 = \begin{pmatrix}
     \xi^x \\ \xi^y \\ \xi^\phi
   \end{pmatrix}
 = \begin{pmatrix}
     -\frac12\int_0^1W_s^2\mathrm ds \\
     \int_0^1W_s\mathrm ds \\
     W_1
   \end{pmatrix}. \]
In this section, we show that it is of Schwartz class, and that $\check u(\check p)>0$ if and only if $2\check y<\check x^2$. As is clear from equation \eqref{eq:defchecku} and the definition \ref{def:Psia} of the spaces $\Psi^a$, it shows directly that $\widetilde u\in\Psi^1$.

\begin{theorem}
\label{th:checkuPsi}
The function $\check u$ is of Schwartz class. In other words, for every multi-index $\alpha$ and integer $k$, the quantity
\[ |\check p|^k\cdot\big|D_{\check p}^\alpha\check u(p)\big| \]
is uniformly bounded in $\check p$. 
\end{theorem}

Although not necessary for our purposes, an explicit expression for its Fourier transform is given in Appendix \ref{app:explicitFourier} below.

\begin{proof}
\emph{Infinite product expansion.} This proof was provided by Konstantin Izyurov as the answer \cite{Izyurov} to a question of the author on Mathoverflow.

Let us compute the characteristic function $\hat u$ of $\xi$; it will be of Schwartz class if and only if $\check u$ is. We can expand the Brownian trajectory as
\[ W_t = \sqrt2\sum_{n\geq0}(-1)^n\omega_n\frac{\sin\big(\pi(n+1/2)t\big)}{\pi(n+1/2)}, \]
for $(\omega_n)_{n\geq0}$ independent variables of distribution $\mathcal N(0,1)$, where the right hand side converges almost surely in $\mathcal C([0,1])$. One can see this by expanding the Hölder-continuous $4$-periodic function
\[ t\mapsto\begin{cases}
              W_{t-4k} & \text{if }4k\leq t<4k+1 \\
              W_{4k+2-t} & \text{if }4k+1\leq t<4k+2 \\
             -W_{t-4k-2} & \text{if }4k+2\leq t<4k+3 \\
             -W_{4k+4-t} & \text{if }4k+3\leq t<4k+4
           \end{cases} \]
in Fourier series. Indeed, the only non-vanishing coefficients are the Gaussian integrals
\[ 2\int_0^1W_t\sin\big(\pi(n+1/2)t\big)\mathrm dt, \]
which are Gaussian variables of variance
\[ \frac2{\pi^2(n+1/2)^2} \]
and with vanishing covariance.

Since all components of the random vector $\xi$ are continuous images of $W$ with respect to the uniform topology, we can define
\begin{align}
\label{eq:defxin}
\xi^x_n    &:= -\frac12\frac{\strut\omega_n^2}{\strut\pi^2(n+1/2)^2}, &
\xi^y_n    &:= \sqrt 2\frac{\strut(-1)^n\omega_n}{\strut\pi^2(n+1/2)^2}, &
\xi^\phi_n &:= \sqrt 2\frac{\strut\omega_n}{\strut\pi(n+1/2)},
\end{align}
and $\xi_n:=(\xi^x_n,\xi^y_n,\xi^\phi_n)$, and have almost surely
\[ \xi = \sum_{n\geq0} \xi_n. \]
In particular, the characteristic function of $\xi$ is the product of the characteristic functions of $\xi_n$, in the pointwise sense.

It is well-known that for $\omega$ a standard Gaussian variable and $\alpha,\beta\in\mathbb C$, $\Re\alpha<1/2$,
\[ \mathbb E[\exp(\alpha\omega^2+\beta\omega)] = \frac{\mathsf e^{\frac{\beta^2}{2-4\alpha}}}{\sqrt{1-2\alpha}}, \]
where the square root is the analytic continuation over $\{\Re z>0\}$ of the usual square root over $\mathbb R^*_+$. This can be seen by noticing that the expectation has to be analytic with respect to $\alpha$ and $\beta$, and computing the integral for $\alpha$ and $\beta$ real. Defining a complex logarithm over the same domain, it means that
\begin{equation}
\label{eq:Fourierxin}
   \mathbb E\left[\mathsf e^{\mathsf i(\lambda,\mu,\nu)\cdot\xi_n}\right]
 = \exp\left( - \frac12\ln\left(1+\frac{\mathsf i\lambda}{\pi^2(n+1/2)^2}\right)
              - \frac{\big(\frac{(-1)^n\mu}{\pi(n+1/2)}+\nu\big)^2}{\pi^2(n+1/2)^2+\mathsf i\lambda}\right) 
\end{equation}
and the characteristic function $\hat u$ is the product of these functions. In other words, $\hat u$ can be expressed as $\exp(f)$, where
\begin{equation}
\label{eq:logFourierxi}
   f(\lambda,\mu,\nu)
 = - \sum_{n\geq0}\frac12\ln\left(1+\frac{\mathsf i\lambda}{\pi^2(n+1/2)^2}\right)
   - \sum_{n\geq0}\frac{\big(\frac{(-1)^n\mu}{\pi(n+1/2)}+\nu\big)^2}{\pi^2(n+1/2)^2+\mathsf i\lambda}.
\end{equation}

Let us state a fact about the regularity of holomorphic exponentials that we will prove later. Let $f:\mathbb R^n\to\mathbb C$ be a function such that
\begin{enumerate}
\item it extends to a holomorphic function $F:\big(\mathbb R+\mathsf i(-\varepsilon,\varepsilon)\big)^n\to\mathbb C$;
\item the extension $F$ grows at most polynomially;
\item the exponential $\exp(-f)$ grows faster than any polynomial.
\end{enumerate}
Then $\exp(f)$ is of Schwartz class.

In our case, note that $f$ extends holomorphically over $\{\zeta:\Im\zeta<\pi^2/4\}\times\mathbb C^2$. Accordingly, we will apply the previous result for some given $\varepsilon<\pi^2/4$. We will prove hypotheses (2) and (3) above in two steps: first, the absolute value of $f$ is bounded by polynomials; then, the real part of $-f$, for large arguments, is bounded below by $(|\lambda|+|\mu|+|\nu|)^c$ for some $c>0$ small enough. From now on we fix an $\varepsilon$ in $(0,\pi^2/4)$.
\medskip

\emph{Bound for the absolute value.}
We will show that the following inequalities hold up to universal constants.
\begin{align*}
&          \left|\sum_{n\geq0}\ln\left(1+\frac{\mathsf i\lambda}{\pi^2(n+1/2)^2}\right)\right|
\lesssim 1+|\lambda|^{1/2}\ln(1+|\lambda|) &
& \text{over }\big(\mathbb R+\mathsf i(-\varepsilon,\varepsilon)\big)^3, \\
&          \left|\sum_{n\geq0}\frac{\big(\frac{(-1)^n\mu}{\pi(n+1/2)}+\nu\big)^2}
                                   {\pi^2(n+1/2)^2+\mathsf i\lambda}\right|
\lesssim |\mu|^2+|\nu|^2 &
& \text{over }\big(\mathbb R+\mathsf i(-\varepsilon,\varepsilon)\big)^3.
\end{align*}

To deal with the first inequality, notice that the sum is smooth over $\big(\mathbb R+\mathsf i[-\varepsilon,\varepsilon]\big)^3$ as a uniform limit of analytic functions, so we can focus on the case $\lambda$ large. For the inequalities we write in the following, the order of the quantifiers is understood as follows: there exists $M,C>0$ such that for all $|\lambda|>M$ in the considered domain, the inequality holds with implicit constant $C$.

For the first inequality, a rough estimate is given by
\[       \left|\ln\left(1+\frac{\mathsf i\lambda}{\pi^2(n+1/2)^2}\right)\right|
    \leq \frac\pi2 + \ln\left|1+\frac{\mathsf i\lambda}{\pi^2(n+1/2)^2}\right|
\lesssim \ln(1+|\lambda|). \]
For large $n$, namely $n^2>|\lambda|$, we have the further estimate
\[       \left|\ln\left(1+\frac{\mathsf i\lambda}{\pi^2(n+1/2)^2}\right)\right|
    \leq \sup_{|\zeta|<1/\pi^2}|\ln'(1+\zeta)|\cdot\left|\frac{\mathsf i\lambda}{\pi^2(n+1/2)^2}\right|
\lesssim \frac{|\lambda|}{n^2}. \]
It means that the sum can be estimated as
\[       \left|\sum_{n\geq0}\ln\left(1+\frac{\mathsf i\lambda}{\pi^2(n+1/2)^2}\right)\right|
\lesssim \sum_{n^2\leq|\lambda|}\ln(1+|\lambda|) + \sum_{n^2>|\lambda|}\frac{|\lambda|}{n^2}
\lesssim |\lambda|^{1/2}\ln(1+|\lambda|). \]

For the second estimate, from the fact that
\[   |\pi^2(n+1/2)^2+\mathsf i\lambda|^2
   = \big(\pi^2(n+1/2)^2-\Im\lambda\big)^2 + (\Re\lambda)^2
\geq (\pi^2(n+1/2)^2-\varepsilon)^2 \]
we get the bound
\[       \left|\sum_{n\geq0}\frac{\big(\frac{(-1)^n\mu}{\pi(n+1/2)}+\nu\big)^2}
                                 {\pi^2(n+1/2)^2+\mathsf i\lambda}\right|
\leq     \sum_{n\geq0}\frac{\frac4{\pi^2}\mu^2+\frac4\pi|\mu\nu|+\nu^2}{\pi^2(n+1/2)^2-\varepsilon}
\lesssim |\mu|^2+|\nu|^2. \]
\medskip

\emph{Lower bound for the real part.}
We consider now the real parts. From now on, all imaginary parts of $\lambda$, $\mu$, $\nu$ are zero. We will show the inequalities
\begin{align*}
&        |\lambda|^{1/2}
\lesssim 1 + \Re\sum_{n\geq0}\ln\left(1+\frac{\mathsf i\lambda}{\pi^2(n+1/2)^2}\right) &
& \text{over }\mathbb R^3, \\
&        \frac{|\mu|^2+|\nu|^2}{1+|\lambda|^2}
\lesssim 1 + \Re\sum_{n\geq0}\frac{\big(\frac{(-1)^n\mu}{\pi(n+1/2)}+\nu\big)^2}
                                  {\pi^2(n+1/2)^2+\mathsf i\lambda} &
& \text{over }\mathbb R^3.
\end{align*}
A lower bound for the real part of $-f$ follows for instance from considering the direction of the inequality between $|\lambda|^2$ and $|\mu|+|\nu|$:
\begin{align*}
          1 - \Re f(\lambda)
& \gtrsim |\lambda|^{1/2} + \frac{|\mu|^2+|\nu|^2}{1+|\lambda|^2} \\
& \geq \begin{cases}
         \frac12|\lambda|^{1/2} + \frac12(|\mu|+|\nu|)^{1/4} + 0 & \text{for }|\lambda|^2\geq|\mu|+|\nu| \\
         |\lambda|^{1/2} + \frac12\cdot\frac{(|\mu|^2+|\nu|^2)^{1/2}|\lambda|^2}{1+|\lambda|^2}
           \gtrsim |\lambda|^{1/2} + |\mu|+|\nu| & \text{for }|\lambda|^2\leq|\mu|+|\nu|.
       \end{cases}
\end{align*}
Since these grow faster than, say, $(|\lambda|+|\mu|+|\nu|)^{1/4}$, $\exp(-f)$ does indeed grow super-polynomially. As in the previous discussion, we focus on the case where the arguments are large.

The real part of a given logarithm rewrites as
\[ \Re\ln\left(1+\frac{\mathsf i\lambda}{\pi^2(n+1/2)^2}\right)
 = \ln\left|1+\frac{\mathsf i\lambda}{\pi^2(n+1/2)^2}\right|
 = \frac12\ln\left(1+\frac{\lambda^2}{\pi^4(n+1/2)^4}\right). \]
In particular, it is always non-negative, so it suffices to show the inequality for a sum over some subset. In fact, for $n^2>|\lambda|^{1/2}$,
\[      \ln\left(1+\frac{\lambda^2}{\pi^4(n+1/2)^4}\right)
   \geq \inf_{0\leq s\leq1/\pi^4}\ln'(1+s)\cdot\frac{\lambda^2}{\pi^4(n+1/2)^4}
\gtrsim \frac{\lambda^2}{n^4} \]
so we can consider only such large terms:
\[      \sum_{n^2\geq|\lambda|}\ln\left(1+\frac{\lambda^2}{\pi^4(n+1/2)^4}\right)
\gtrsim \sum_{n^2\geq|\lambda|}\frac{\lambda^2}{n^4}
\gtrsim \lambda^{1/2}. \]

Finally, we consider the real part of the quadratic form. Since
\[   \Re\left(\frac1{\pi^2(n+1/2)^2+\mathsf i\lambda}\right)
   = \frac{\Re(\pi^2(n+1/2)^2-\mathsf i\lambda)}{|\pi^2(n+1/2)^2+\mathsf i\lambda|^2}
\geq \frac{\frac{\pi^2}4}{\pi^4(n+1/2)^4+\lambda^2}
   > 0, \]
it suffices to know that the sum over a finite subset grows at least linearly. But it is already the case considering the first two terms:
\begin{align*}
    \Re\sum_{n=0,1}\frac{\big(\frac{(-1)^n\mu}{\pi(n+1/2)}+\nu\big)^2}{\pi^2(n+1/2)^2+\mathsf i\lambda}
& = \frac{\frac{\pi^2}4}{\big(\frac{\pi^2}4\big)^2+\lambda^2}\left(\frac{2\mu}{\pi}+\nu\right)^2
  + \frac{\frac{9\pi^2}4}{\big(\frac{9\pi^2}4\big)^2+\lambda^2}\left(\frac{-2\mu}{3\pi}+\nu\right)^2 \\
& \gtrsim\frac1{1+\lambda^2}\left(\left(\frac{2\mu}{\pi}+\nu\right)^2+\left(\frac{-2\mu}{3\pi}+\nu\right)^2\right) \\
& \gtrsim\frac{|\mu|^2+|\nu|^2}{1+\lambda^2}.
\end{align*}
This will conclude, provided we prove the above result about the regularity of exponentials.
\medskip

\emph{Regularity of holomorphic exponentials.}
Let $F:\big(\mathbb R+\mathsf i(-\varepsilon,\varepsilon)\big)^n\to\mathbb C$ be holomorphic with polynomial growth, and suppose $\exp(-F)$ has super-polynomial growth over $\mathbb R^n$.

Fix $d>0$ such that $|F(z)|\lesssim1+|z|^d$. The crucial observation is that using the Cauchy formula, for any $z\in\big(\mathbb R+\mathsf i(-\varepsilon/2,\varepsilon/2)\big)^n$ and $\epsilon_i$ in the canonical basis, 
\[       |\partial_iF(z)|
       = \left|\frac1{2\mathsf i\pi}
               \int_{\partial B_{\varepsilon/2}}\frac{F(z+w\epsilon_i)}{w^2}\mathrm dw\right|
    \leq \frac2\varepsilon\cdot\sup_{|w|<\varepsilon/2}|F(z+w)|
\lesssim \frac{1+|z|^d}\varepsilon. \]
We then see by induction that any given derivative of $F$ has the same growth order $O(|z|^d)$ as the initial function. Note that a better constant would be given by using multiple integrals and higher powers in the Cauchy formula, which we do not need in our case.

One concludes by noticing that each derivative of $\exp(F)$ is a product of $\exp(F)$ by a polynomial in the derivatives of $F$: the latter grows at most polynomially, while the former decreases fast enough to compensate the growth of any polynomial.
\end{proof}

Although not directly necessary for the proof of the main theorem, it is of interest to get a better understanding of the support of $\check u$; we will use it in the proof of Corollary \ref{cor:multiplicativebigO}. The next two propositions explain where the mass is concentrated.

\begin{proposition}
\label{prop:support}
The support of $\check u$ is $\{\check p:\check y^2\leq-2\check x\}$.
\end{proposition}

\begin{proof}
Almost surely, we have
\[ b^2 = \Big(\int_0^1W_s\mathrm ds\Big)^2 \leq \int_0^1W_s^2\mathrm ds = -2a, \]
so the support is included in $\{\check p:\check y^2\leq-2\check x\}$. In the opposite direction, we will use the classical Stroock-Varadhan support theorem in the following (very weak) form: if $w:[0,1]\to\mathbb R$ starting from zero is smooth, then the point
\[ \left(-\int_0^1\frac{w(s)^2}2\mathrm ds, \int_0^1w(s)\mathrm ds,w(1)\right) \]
is in the support of $\check u$ --- see for instance the original \cite[Theorem 5.2]{StroockVaradhan}.

Let $\check p$ be a point such that $\check y^2<-2\check x$. Let $w_n$ be a sequence of smooth functions from $[0,1]$ to $\mathbb R$ such that $w_n(0)=0$, $w_n(1)=\check\phi$, and $w_n\to w$ in the $L^1$ and $L^2$ topologies, where
\[w(t):=\frac{\check y}r\mathbf 1_{t<r}, \qquad r := \frac{\check y^2}{-2\check x}<1. \]
Then, according to the Stroock-Varadhan theorem,
\[ \left(-\int_0^1\frac{w_n(s)^2}2\mathrm ds, \int_0^1w_n(s)\mathrm ds,w_n(1)\right)\in\mathrm{supp}(\check u), \]
hence $\check p$ is in the support as well and so is all of $\{\check p:\check y^2\leq-2\check x\}$.
\end{proof}

\begin{proposition}
\label{prop:strictsupport}
We have $\check u(\check p)>0$ if and only if $\check y^2<-2\check x$.
\end{proposition}

\begin{proof}
It is clear from the support Proposition \ref{prop:support} that the direct implication must be true. It remains to show that a given $\check p$ with $\check y^2<-2\check x$ not only belongs to the support, but is in fact a point of positive density.

For the purposes of this proof, it is useful to get back to the dynamics involved. Recall that according to equation \eqref{eq:defchecku}, $\check u$ is a translate of $\widetilde u_1(0,\cdot)$, whereas $\widetilde u_t(p_0,\cdot)$ is the density of $p_t$, for $p$ the solution of equation \eqref{eq:tildeSDE}. Since the differential equation is time-independent, we can write $\widetilde u_1$ as the convolution of $\widetilde u_{1/2}$ with itself; explicitly, we have
\begin{align*}
   \check u(\check p)
&= \int_{\mathbb R^3} \widetilde u_{1/2}\big(0,p_*+(1/2,0,0)\big) \cdot
                      \widetilde u_{1/2}\big(p_*+(1/2,0,0),\check p+(1,0,0)\big)\mathrm dp_* \\
&= \int_{\mathbb R^3} \widetilde u_{1/2}\big(0,p_*+(1/2,0,0)\big) \cdot
                      \widetilde u_{1/2}\big(0,M_{p_*}^{-1}(\check p-p_*+(1/2,0,0))\big)\mathrm dp_*,
\end{align*}
where in the second line we used the invariance property \eqref{eq:defchecku}. Based on the representation \eqref{eq:tildepintegral}, we can use the same argument as in Proposition \ref{prop:support} to see that the support of $\widetilde u_{1/2}(0,\cdot)$ is $\{p:y^2\leq-x+1/2\}$.

Consider a point $p_*$ close to $(\check x/2,\check y/2,0)$. Because of the condition on $\check p$, we know that
\[ \Big(\frac12\check y\Big)^2<-\frac12\check x, \]
so both $p_*+(1/2,0,0)$ and $M_{p_*}^{-1}(\check p-p_*+(1/2,0,0))$ belong to the interior of the support of $\widetilde u_{1/2}(0,\cdot)$ for all $p_*$ in a small neighbourhood of $(\check x/2,\check y/2,0)$. It means that in this neighbourhood, the points such that $u_{1/2}(0,\cdot)$ is zero at either of these two points is the union of two closed sets of empty interior, so it must be of empty interior itself. In other words, defining $U$ as the open set of the $p_*$ such that we have
\begin{align*} 
\widetilde u_{1/2}\big(0,p_*+(1/2,0,0)\big)&>0 &
& \text{and} &
\widetilde u_{1/2}\big(p_*+(1/2,0,0),\check p+(1,0,0)\big) &> 0,
\end{align*}
the closure of $U$ contains a neighbourhood of $(\check x/2,\check y/2,0)$. In particular, this open set is not empty, and we see that the integral has to be positive:
\[ \check u(\check p)
 = \int_U \widetilde u_{1/2}\big(0,p_*+(1/2,0,0)\big) \cdot
            \widetilde u_{1/2}\big(p_*+(1/2,0,0),\check p+(1,0,0)\big) \mathrm dp_*
 > 0. \qedhere \]
\end{proof}

\section{Convergence of the series}
\label{sec:convergenceseries}

We now have the formal tools to prove the convergence of the series defined in equation \eqref{eq:Duhamelseries} as a candidate for the density $u$. The objective of this section is the following theorem.

\begin{theorem}
\label{th:seriesconvergence}
For all $T>0$ and $\ell\in\mathbb N$, the formal sum
\[ \widetilde u + E*\widetilde u + (E*)^2\widetilde u + \cdots \]
is in fact convergent in the $\mathcal C^\ell$ topology, uniformly over the sets $[0,T]\times\mathbb R^3\times\mathbb R^3$.
\end{theorem}

The convergence is to be understood in the following sense: even if the first few terms are not $\mathcal C^\ell$ (or even continuous) at the boundary $t=0$, the series $\sum_{k>K}(E*)^k\widetilde u$ eventually belongs to $\mathcal C^\ell([0,T]\times\mathbb R^3\times\mathbb R^3)$ and converges to zero in this topology as $K\to\infty$.

We begin the proof with a qualitative bound on the error. Recall that $E=(-\partial_t+L^*)\widetilde u$, for $L$ the generator of kinetic Brownian motion as defined in equation \eqref{eq:trueSDE}.

\begin{proposition}
\label{prop:EinPsi}
The error $E$ belongs to $\Psi^1$.
\end{proposition}

Let us stress that the expected regularity for $E$, knowing only that $\widetilde u\in\Psi^1$, should be $E\in\Psi^{-1}$ (from the derivative in the direction $X$). This is the main ingredient in our approach: being able to find some balance between some $\widetilde u$ simple enough so that we can study its regularity, and matching the equation well enough that the error is regular.

\begin{proof}
Note that since $\widetilde u$ satisfies the Fokker-Planck equation \ref{eq:FokkerPlanck}, we have
\[ E = (-\partial_t+L^*)\widetilde u
     = (L^*-\widetilde L_{p_0}^*)\widetilde u
     = (-X+\widetilde X)\widetilde u. \]
Because of the representation of $\widetilde u$ involving $\check u$, it is convenient to write these two vector fields in the basis $M_{p_0}$.
\[ M_{p_0}^{-1}(\widetilde X - X)(p)
 = \begin{pmatrix}
     1-\frac{(\phi-\phi_0)^2}2-\cos(\phi-\phi_0) \\
     (\phi-\phi_0) - \sin(\phi-\phi_0) \\
     0
\end{pmatrix}. \]
On the other hand, taking the derivative with respect to $\check x$ and $\check y$ in equation \eqref{eq:defchecku} yields
\begin{align*}
   \partial_{\check x}\check u(\check p)
&= \tau^{8+4}(X(p_0)\cdot\widetilde u)_{\tau^2}(p_0,P_{\tau,p_0}(\check p)), \\
   \partial_{\check y}\check u(\check p)
&= \tau^{8+3}(Y(p_0)\cdot\widetilde u)_{\tau^2}(p_0,P_{\tau,p_0}(\check p)).
\end{align*}
Setting $\mathcal C$ (resp. $\mathcal S$) the analytic continuations of $(\cos(x)-1+x^2/2)/x^4$ (resp. $(\sin(x)-x)/x^3$), the above gives
\[ \tau^8E_{\tau^2}\big(p_0,P_{\tau,p_0}(\check p)\big)
 = - \mathcal C(\tau\check\phi)\check\phi^4(\partial_{\check x}\check u)(\check p)
   - \mathcal S(\tau\check\phi)\check\phi^3(\partial_{\check y}\check u)(\check p). \]
Because $\check u$ is in $\check\Psi$ and the derivatives of $\mathcal C$ and $\mathcal S$ are bounded by polynomials, $E$ does indeed belong to $\Psi^1$.
\end{proof}

We can now turn to the proof of Theorem \ref{th:seriesconvergence}.

\begin{proof}[Proof of Theorem \ref{th:seriesconvergence}.]
Recall that $\mathcal T$ is the operator of multiplication by the function $t$. We show that the following stronger result holds. For any $A\in\Psi^a$, $a>0$, and any $\alpha_0,\alpha\in\mathbb N^3$, $n,m\in\mathbb N$, there exists $k_0\geq0$ such that each term of the series
\begin{equation}
\label{eq:stronglyconvergentseries}
\sum_{k\geq k_0}\mathcal T^{-m}(\partial_t\mathcal T)^nD_{p_0}^{\alpha_0}D_p^\alpha((E*)^kA)
\end{equation}
extends to a function continuous up to the boundary, and the series converges uniformly over the sets $[0,T]\times\mathbb R^3\times\mathbb R^3$. It is indeed stronger, since the operator $(\partial_t\mathcal T)^n$ decomposes as a linear combination of $\mathcal T^j\partial_t^j$ for $j\leq n$, with leading coefficient $\mathcal T^n\partial_t^n$ (the coefficients are the Stirling numbers of the second kind). Hence, $\partial_t^n$ is a linear combination of terms of the form $\mathcal T^{-n}(\partial_t\mathcal T)^j$, for $j\leq n$ (the coefficients are the signed Stirling numbers of the first kind).

Fix exponents $\alpha_0,\alpha\in\mathbb N^3$ and $n\in\mathbb N$, as well as a time horizon $T>0$ and a function $A\in\Psi^a$, $a>0$. Set
\[ K = n + \max(\alpha_0^x + \alpha_0^y + \alpha_0^\phi,\alpha^x + \alpha^y + \alpha^\phi). \]
Then there exists a finite set of functions $\mathcal A\subset\Psi^{5+2K}$ and a constant $c>0$ such that for all $k$ large enough, one can write
\[ (E*)^kA = A_0*\cdots*A_\ell \]
with $A_i\in\mathcal A$ and $\ell\geq\max(1,k/c)$. For instance, one can take
\[ \mathcal A = \big\{ (E*)^kU, 5+2K\leq k<10+4K \big\} \cup \big\{E^{*(5+2K)}\big\}. \]
For functions in $\mathcal A$, Proposition \ref{prop:derivativesconvolution} and its following remark, by induction, yield
\begin{multline*}
\mathcal T^{-m}(\partial_t\mathcal T)^nD_{p_0}^{\alpha_0}D_p^\alpha(A_0*\cdots*A_\ell) \\
= \sum_{n_0+\cdots+n_\ell=n}\binom{n}{n_0,\cdots,n_\ell}\mathcal T^{-m}\big[(\partial_t\mathcal T)^{n_0}(D_{p_0}^{\alpha_0}A_0)*\cdots*(\partial_t\mathcal T)^{n_i}A_i*\cdots*(\partial_t\mathcal T)^{n_\ell}(D_p^\alpha A_\ell)\big].
\end{multline*}
Because $\mathcal A$ is finite, the functions $(\partial_t\mathcal T)^j(D_{p_0}^{\alpha_0}A')$, $(\partial_t\mathcal T)^j A'$ and $(\partial_t\mathcal T)^j(D_p^\alpha A')$ form a finite family for $j\leq n,A'\in\mathcal A$. Moreover, they all belong to $\Psi^5$ according to Proposition~\ref{prop:lossofregularity} since $\Psi^{5+2K}$. Using Proposition \ref{prop:regularconvolution}, there exists a constant $M>0$ such that each term in the above sum is bounded by
\[ \binom{n}{n_0,\cdots,n_\ell}\frac1{\ell!}M^{\ell+1}T^{5\ell-m} \]
whenever $5\ell-m\geq0$.

For $k$ sufficiently large, we decompose $(E*)^kA$ into a convolution product of $\ell+1$ functions in $\mathcal A$, where $\ell\geq\max(1,k/c)$ is large as well. Then
\begin{multline*}
     \big|\mathcal T^{-p}(\partial_t\mathcal T)^nD_{p_0}^{\alpha_0}D_p^\alpha((E*)^kA)\big|_\infty \\
\leq \sum_{n_0+\cdots+n_\ell=n}\binom{n}{n_0,\cdots,n_\ell}\frac1{\ell!}M^{\ell+1}T^{5\ell-m}
   = \frac1{\ell!}(\ell+1)^nM^{\ell+1}T^{5\ell-m}.
\end{multline*}
In particular, since $\ell\geq k/c$, there exists some $k_0\geq0$ such that the sum over $k\geq k_0$ extends to a series of continuous functions over $[0,T]\times\mathbb R^3\times\mathbb R^3$ that converges uniformly (at rate at least $\exp(-\varepsilon k\ln k)$ for some $\varepsilon>0$).
\end{proof}

\section{Fokker-Planck equation}
\label{sec:FokkerPlanck}

To prove that $u$ is the sum $U$ of the Duhamel series \eqref{eq:Duhamelseries}, we will show that $U$ satisfies some version of the Fokker-Planck equations \eqref{eq:FokkerPlanck}, and prove a uniqueness result. Let us sketch the uniqueness argument, and see what kind of properties of $U$ we actually need. It is in spirit a forward-backward argument, so we will use the following lemma to construct a backward diffusion. See section \ref{ssec:timereversal} for a reference.
\begin{lemma}
\label{lem:timereversal}
The generator $L$ formally preserves the Lebesgue measure, in the sense that $L^*1=0$, so its adjoint is associated to a diffusion $p^*$. This diffusion is itself Feller and hypoelliptic, so it admits a kernel $u^*$. This kernel is the time-reversal of $u$, in the sense that
\[ \int_{(\mathbb R^3)^2}f(x)u^*_t(x,y)g(y)\mathrm dx\mathrm dy
 = \int_{(\mathbb R^3)^2}g(x)u_t(x,y)f(y)\mathrm dx\mathrm dy \]
for every $f$ and $g$ continuous with compact support.
\end{lemma}

Let $p^*$ be defined as such. Let us fix a time horizon $T>0$, an initial point $p_0$, and a smooth probability measure $\mathrm d\mu=g(p)\mathrm dp$ on $\mathbb R^3$ with compact support (the initial condition for $p^*$). Since $U$ is smooth, we know that
\[ t\mapsto U_{T-t}(p_0,p^*_t)-\int_0^t\big((-\partial_t+L^*)U\big)_{T-s}(p_0,p^*_s)\mathrm ds \]
is a local martingale for $0\leq t<T$ and $\mathcal L(p^*_0)=\mu$. Since $U$ also has bounded derivatives when $t$ ranges over compact intervals of $(0,\infty)$, it is actually a martingale over $[0,T)$.

Suppose we have shown that $(-\partial_t+L^*)U=0$ (in the pointwise sense, but since $U$ is smooth it suffices to know that it holds in the distributional sense). This will be the point of Theorem \ref{th:FokkerPlanck}. Then
\[ \mathbb E_\mu[U_t(p_0,p^*_\varepsilon)] = \mathbb E_\mu[U_\varepsilon(p_0,p^*_t)], \]
or in other words
\begin{equation}\label{eq:utoU}
   \int g(p^*_0)u^*_\varepsilon(p^*_0,p^*)U_t(p_0,p^*)\mathrm dp^*_0\mathrm dp^*
 = \int g(p^*_0)u^*_t(p^*_0,p^*)U_\varepsilon(p_0,p^*)\mathrm dp^*_0\mathrm dp^*.
\end{equation}

If we could take the limit $\varepsilon\to0$, we would get formally
\begin{align*}
   \int g(p)U_t(p_0,p)\mathrm dp
&= \int g(p^*_0)\delta_{p^*_0}(p^*)U_t(p_0,p^*)\mathrm dp^*_0\mathrm dp^* \\
&= \int g(p^*_0)u^*_t(p^*_0,p^*)\delta_{p_0}(p^*)\mathrm dp^*_0\mathrm dp^*
 = \int g(p)u^*_t(p,p_0)\mathrm dp.
\end{align*}
Rigorously, since $p^*\mapsto U_t(p_0,p^*)$ is continuous and vanishes at infinity, and since the diffusion is Feller, we know that
\[ \big(P^*_\varepsilon U_t(p_0,\cdot)\big)(p^*_0)
 = \int u^*_\varepsilon(p^*_0,p^*)U_t(p_0,p^*)\mathrm dp^*\to U_t(p_0,p^*_0), \]
and that the convergence actually holds in the uniform topology as functions of $p^*_0$. Since $g$ has compact support, it means that we can take the limit in the left hand side of \eqref{eq:utoU}. As for the right hand side, because of Lemma \ref{lem:timereversal}, we can replace $u^*$ by $u$, and we know from the Feller property for the forward diffusion that
\[ P_tg:p^*_0\mapsto\int u_t(p^*_0,p^*)g(p^*)\mathrm dp^* \]
is continuous and vanishes at infinity. We can actually show, from elementary properties of the diffusion, that it decreases very fast, definitely fast enough to be integrable; see Proposition \ref{prop:exponentialFeller}. Then the right hand side of \eqref{eq:utoU} rewrites as
\[ \int U(p_0,p)P_tg(p)\mathrm dp, \]
and we can prove the convergence provided $U$ acts as an approximation of unity for continuous integrable functions. This is the point of Proposition \ref{prop:initialcondition}, and assuming this result we have established
\[ \int U_t(p_0,p)g(p)\mathrm dp = \int u_t(p_0,p)g(p)\mathrm dp. \]
Since $U$ and $u$ are smooth as functions of their second variables, this concludes the proof of $u=U$.

This shows that a viable strategy is to prove that $U$ satisfies $\partial_tU=L^*U$, and that $U$ is in some sense an approximation of unity. In this section, we set to prove these two results, in the form of Theorem \ref{th:FokkerPlanck} and Proposition \ref{prop:initialcondition}.

\begin{theorem}
\label{th:FokkerPlanck}
For every given $p_0\in\mathbb R^3$, the function $U:(t,p)\mapsto U_t(p_0,p)$ satisfies
\[ \partial_tU = L^*U \]
over the open set $\mathbb R_+^*\times\mathbb R^3$.
\end{theorem}

\begin{proposition}
\label{prop:initialcondition}
Let $f:\mathbb R^3\to\mathbb R$ be an integrable function, and suppose it is continuous at $p_0$. Then the space convolution
\[ (U_\varepsilon\star f)(p_0) = \int_{\mathbb R^3}U_\varepsilon(p_0,p_*)f(p_*)\mathrm dp_* \]
converges to $f(p_0)$ as $\varepsilon$ goes to zero.
\end{proposition}

The convergence actually holds uniformly in a neighbourhood of $p_0$ if $f$ is continuous in such a neighbourhood, but we will not have to use this stronger result.

Before we get to the proofs, let us state a lemma saying in essence that right convolution by $\widetilde u_t$ behaves roughly as a continuous semigroup.

\begin{lemma}[proved in Appendix \ref{app:Fellerandco}, section \ref{ssec:convolutionwu}]
\label{lem:convolutionwu}
Let $f:\mathbb R^3\to\mathbb R$ be bounded measurable. Then for all $t>0$, space convolution with $\widetilde u$ is a contraction:
\[ |f\star\widetilde u_t|_\infty\leq|f|_\infty. \]
If moreover $f$ is uniformly Lipschitz, then $f\star\widetilde u_t$ converges uniformly to $f$ as $t$ goes to zero, in the following strong sense:
\[ \lim_{t\to0}\sup_{|f|_\mathrm{Lip}\leq1}|f\star\widetilde u_t-f|_\infty=0. \]
\end{lemma}

\begin{proof}[Proof of Theorem \ref{th:FokkerPlanck}]
We show that $(\partial_t-L^*)((E*)^{k+1}\widetilde u) = (E*)^kE - (E*)^{k+1}E$. The sum then telescopes to give, say in the pointwise sense,
\[ (\partial_t-L)U
 = -E + \sum_{k\geq0}\big((E*)^kE - (E*)^{k+1}E\big)
 = \lim_{k\to\infty}(E*)^kE. \]
Note that we used the fact that the series defining $U$ converges in the $\mathcal C^2$ topology. Since, as stated at the beginning of the proof of Theorem \ref{th:seriesconvergence}, the series $\sum_{k\geq0}(E*)^kE$ is convergent in $\mathcal C^0$, the limit is in fact zero and $U$ satisfies the partial differential equation.

It remains to show the property stated above. We will prove a slightly more general result: for any $A\in\Psi^a$, $a>0$,
\[
\partial_t(A*\widetilde u) = A + L^*(A*\widetilde u) - A*E. \]
For $A=(E*)^kE$, this gives as expected
\[ \partial_t\big((E*)^{k+1}\widetilde u\big)
 = (E*)^kE + L^*((E*)^{k+1}\widetilde u) - (E*)^{k+1}E. \]

Write $\star$ for the space convolution; for instance,
\[ A_s\star\widetilde u_{t-s}:
   (p_0,p)\mapsto\int_{\mathbb R^3}A_s(p_0,p_*)\widetilde u_{t-s}(p_*,p)\mathrm dp_*. \]
Recall also that $A\pconv_\varepsilon^{1-\varepsilon}B$ denotes the space-time convolution defined in Lemma \ref{lem:new-uniformab}. The essence of the reasoning is contained in the following formal derivation:
\begin{gather}
\label{eq:timederivativeAuepsilon}
\begin{aligned}
\partial_t\big(A\pconv_\varepsilon^{1-\varepsilon}\widetilde u\big)_t
&= \frac\partial{\partial t}\Big(\int_{\varepsilon t}^{(1-\varepsilon)t}A_s\star\widetilde u_{t-s}\,\mathrm ds\Big) \\
&= (1-\varepsilon)A_{(1-\varepsilon)t}\star\widetilde u_{\varepsilon t}
 - \varepsilon A_{\varepsilon t}\star\widetilde u_{(1-\varepsilon)t}
 + \big(A\pconv_\varepsilon^{1-\varepsilon}\partial_t\widetilde u\big)_t \\
&= (1-\varepsilon)A_{(1-\varepsilon)t}\star\widetilde u_{\varepsilon t}
  - \varepsilon A_{\varepsilon t}\star\widetilde u_{(1-\varepsilon)t} \\
&\qquad + L^*\big(A\pconv_\varepsilon^{1-\varepsilon}\widetilde u\big)_t
  - \big(A\pconv_\varepsilon^{1-\varepsilon}E\big)_t \\
& =: (*)_1 + (*)_2 + (*)_3 + (*)_4.
\end{aligned}
\end{gather}
The rest of the proof consists in showing that the relation holds, and that it converges in some weak sense to the expected expression for $\partial_t(A*\widetilde u)_t$.
\medskip

\emph{Expression of the derivative.}
Consider a real-valued function $(t,s)\mapsto f(t,s)$. Elementary computations show that to have the representation
\[ \frac\partial{\partial t}\Big(\int_{\varepsilon t}^{(1-\varepsilon)t}f(t,s)\,\mathrm ds\Big)
 = (1-\varepsilon)f\big(t,(1-\varepsilon)t\big)
 - \varepsilon f(t,\varepsilon t)
 + \int_{\varepsilon t}^{(1-\varepsilon)t}(\partial_tf)(t,s)\,\mathrm ds, \]
for all $t>0$, it is (more than) sufficient for $\partial_tf$, $\partial_sf$ and $\partial_t^2f$ to be locally bounded over $\{0<s<t\}$. We are interested in the case where $f$ is of the form
\[ (t,s)\mapsto \big(A_s\star\widetilde u_{t-s}\big)(p_0,p) = \int_{\mathbb R^3}A_s(p_0,p_*)\widetilde u_{t-s}(p_*,p)\mathrm dp_*. \]
Note that all derivatives of $A$ and $\widetilde u$ belong to some $\Psi$ space, and for every $B\in\Psi^b$, $b\in\mathbb R$ we have
\[ |B_t(p_0,p_*)| + |B_t(p_*,p)| \leq\frac{C(t,p_0,p)}{1+|p_*|^2}, \]
for some $C(t,p_0,p)>0$ locally uniform. Accordingly, derivatives of $A\star\widetilde u$ commute with the integral, and satisfy the expected bounds.

To get to the last line of \eqref{eq:timederivativeAuepsilon}, we notice that $\partial_t\widetilde u=L^*\widetilde u-E$ by definition of $E$. Hence it remains to show that $L^*$ commutes with partial convolution, but this is clear from Lemma \ref{lem:new-uniformab} and the fact that all partial derivatives of $L^*\widetilde u$ belong to some $\Psi$ space (see Proposition \ref{prop:lossofregularity}).
\medskip

\emph{Identification of the terms.}
Consider all quantities as functions of $(p_0,p)$. We will show that the four terms $(*)_1$ to $(*)_4$ in the last line of equation \eqref{eq:timederivativeAuepsilon} converge respectively to $A$, $0$, $L^*(A*\widetilde u)$ and $A*E$ in distribution (or, say, in the dual of $\mathcal C^2_c(\mathbb R^3\times\mathbb R^3)$) as $\varepsilon$ goes to zero.

According to Lemma \ref{lem:new-uniformab}, partial convolution converges locally uniformly to full convolution when both terms are in some $\Psi^b$, $b>0$. It means that $(*)_3$ and $(*)_4$ converge to their expected limits, respectively in the dual topologies of $\mathcal C^2_c$ and $\mathcal C^0_c$. One can see that $(*)_1$ converges to $A_t$ as a consequence of Lemma \ref{lem:convolutionwu}, since $A_{(1-\varepsilon)t}$ converges uniformly to $A_t$ and the Lipschitz norm of $p\mapsto A_t(p_0,p)$ is uniformly bounded as a function of $p_0$ (both facts being a consequence of $A\in\Psi^a$).

It remains to deal with $(*)_4$. Taking advantage of the fact that $a>0$, we will show that for every $\rho:\mathbb R^3\to\mathbb R$ continuous with compact support, $(\varepsilon t)^{1-a}\rho\star A_{\varepsilon t}\star\widetilde u_{(1-\varepsilon)t}$ is uniformly bounded, as $\varepsilon$ goes to zero. Note in passing that this convolution is associative. From Lemma \ref{lem:convolutionwu}, it is enough to show that $(\varepsilon t)^{1-a}\rho\star A_{\varepsilon t}$ is uniformly bounded. But from $A\in\Psi^a$, we know there exists a constant $C>0$ such that for all $\varepsilon$ small enough,
\begin{align*}
     \big|(\rho\star A_{\varepsilon t})(p_0,p)\big|
&\leq \int_{\mathbb R^3}|\rho(p_*)|\cdot \big|A_{\varepsilon t}(p_*,p)\big|\mathrm dp_* \\
&\leq \int_{\mathbb R^3}\big|\rho\big(P_{-\sqrt{\varepsilon t},p}(\check p_*)\big)\big|\cdot \frac{C(t\varepsilon)^{-5+a}}{1+|\check p_*|^4}(\varepsilon t)^4\mathrm d\check p_* \\
&\leq C|\rho|_\infty(\varepsilon t)^{a-1}\int_{\mathbb R^3}\frac{\mathrm d\check p_*}{1+|\check p_*|^4}.
\end{align*}
Multiplying by $(\varepsilon t)^{-1+a}$, we get a bounded quantity, and the proof of Theorem \ref{th:FokkerPlanck} is complete.
\end{proof}

\begin{proof}[Proof of Proposition \ref{prop:initialcondition}]
Since $f$ is integrable and the series
\[ U=\sum_{k\geq0}(E*)^k\widetilde u \]
converges uniformly, in the sense that the sum for $k\geq K$ tends to 0 as $K$ goes to infinity, uniformly over $[0,1]\times\mathbb R^3\times\mathbb R^3$ (although the supremum is infinite for small values of $K$), it suffices to show that
\begin{align*}
(\widetilde u_\varepsilon\star f)(p_0)&\to f(p_0), &
\big(((E*)^k\widetilde u)_\varepsilon\star f\big)(p_0)&\to 0
\end{align*}
for all fixed $k>0$ as $\varepsilon$ goes to 0. The second convergence will follow if we show $(A_\epsilon\star f)(p_0)\to0$ whenever $A\in\Psi^a$ for $a>1$. We fix such a function $A$. By linearity, we can use a bump function and consider two cases where $f$ is either bounded and continuous at $p_0$, or integrable and zero on a neighbourhood of $p_0$.

\emph{Case 1: bounded continuous at $p_0$.}
Since $\widetilde u(t,p_0,\cdot)$ is the density of a continuous process starting at $p_0$, the first convergence follows from dominated convergence using the bound $\|f\|_\infty\in L^1(\widetilde{\mathbb P})$; namely,
\[ (u_\varepsilon\star f)(p_0)
 = \widetilde{\mathbb E}_{p_0}[f(p_\varepsilon)]
 \to\widetilde{\mathbb E}_{p_0}[f(p_0)]
 = f(p_0). \]
For the second term we choose a constant $M>0$ such that
\[ \left|A_t\big(p_0,P_{\sqrt t,p_0}(\check p)\big)\right|\leq\frac1{t^{4-a}}\cdot\frac M{1+|\check p|^4}. \]
Such a constant exists by definition of $\Psi^a$. Then
\begin{align*}
      \big|(A_\varepsilon\star f)(p_0)\big|
  & = \left|\int_{\mathbb R^3}A_\varepsilon(p_0,p_*)f(p_*)\mathrm dp_*\right| \\
  & = \left|\int_{\mathbb R^3}A_\varepsilon\big(p_0,P_{\sqrt\varepsilon,p_0}(\check p_*)\big)
                              f\big(P_{\sqrt\varepsilon,p_0}(\check p_*)\big)
                              \varepsilon^4\mathrm d\check p_*\right| \\
&\leq \varepsilon^4\int_{\mathbb R^3}\frac1{\varepsilon^{4-a}}\frac{M\|f\|_\infty}{1+|\check p_*|^4}\mathrm d\check p_*,
\end{align*}
which vanishes in the limit since $a>0$.

\emph{Case 2: integrable with 0 outside the support.}
Suppose $f(p)$ is zero for $|p|<\delta$. The proof is actually very similar in this case; where we used integrability properties of the left side of the convolution and boundedness properties of the right one, we do the opposite here. We begin with an observation. By definition of $P_{\tau,p_0}$, we have
\[ \big|P_{\tau,p_0}(\check p)-p_0\big|\leq\varepsilon + \sqrt\varepsilon|\check p|, \]
so if the left hand side is larger than $\delta$ we must have $|\check p|$ larger than $\delta/(2\sqrt\varepsilon)$ for all $\varepsilon$ small enough.

Let us turn to the $\widetilde u\star f$ term. If $f$ is zero over a ball of radius $\delta>0$ around zero,
\begin{align*}
      \big|\widetilde u_\varepsilon(p_0,p_*)f(p_*)\big|
&\leq \sup_{|p-p_0|\geq\delta}\big|\widetilde u_\varepsilon(p_0,p)\big|\cdot |f(p_*)| \\
&   = \sup\left\{\varepsilon^{-4}\big|\check u(\check p)\big|:
                 \check p\text{ such that }|P_{\sqrt\varepsilon,p_0}(\check p)-p_0|\geq\delta\right\}
      \cdot|f(p_*)| \\
&   = \varepsilon^{-4}\cdot\sup_{|\check p|\geq\delta/(2\sqrt\varepsilon)}\big|\check u(\check p)\big|
      \cdot |f(p_*)|.
\end{align*}
Since $\check u$ is of Schwartz class, this supremum decreases faster than polynomially in $\varepsilon$; for instance, it is at most $M(\sqrt\varepsilon)^9$ for some constant $M>0$ depending on $\delta$. This shows that the left hand side is bounded by $\varepsilon^{1/2}|f|$ up to a constant factor, uniformly in $\varepsilon$. Since $|f|$ is integrable, $(\widetilde u_\varepsilon\star f)(p_0)$ must go to zero (faster than any polynomial).

The term $A\star f$ is treated in the same manner, noticing that
\[ \sup_{|\check p|\geq\delta/(2\sqrt\varepsilon)}\big|
   A\big(\sqrt\varepsilon,p_0,P_{\sqrt t,p_0}(\check p)\big)\big| \]
must also decrease faster than polynomially as $\varepsilon$ goes to zero, since $A\in\Psi^a$. In fact we could have considered any $a\in\mathbb R$, and in which case both terms would have been treated at once.
\end{proof}

This concludes this section, and the proof that the kernel $u$ admits the representation as the Duhamel series \eqref{eq:Duhamelseries}.

\section{Small time expansions}
\label{sec:corollaries}

The representation of $u$ as the Duhamel series, as stated the main Theorem \ref{mth:Duhamel}, holds for all times. From it, we can extract various expansions for small time. We start by estimating the remainder of the Duhamel series itself, and then we present expressions with time-independent coefficients in Corollaries \ref{cor:tauexpansion} and \ref{cor:multiplicativebigO}, respectively in the additive and multiplicative sense.

\begin{corollary}
\label{cor:additivebigO}
The Duhamel series for $u$ is actually an asymptotic expansion for small $t$. In precise terms,
\[ u_t(p_0,p)
 = \sum_{k=0}^K\big((E*)^k\widetilde u\big)_t(p_0,p) + O_K\big(t^{-3+K}\big), \]
where the remainder term is uniformly bounded by $t^{-3+K}$ for $t\leq1$, up to a constant depending only on $K$.
\end{corollary}

Actually, it will be clear from the proof that the expansion also holds for all derivatives $\partial_t^nD_{p_0}^{\alpha_0}D_p^\alpha u$ of $u$, with a different order for the remainder.

\begin{proof}
From the proof of Theorem \ref{th:seriesconvergence}, we know that the Duhamel series still converges uniformly when we add a negative power of $t$ as a prefactor. In particular, there exists an index $K'$ such that
\[ \left|t^{-3+K}u_t-\sum_{k=0}^{K'}t^{-3+K}\big((E*)^k\widetilde u\big)_t(p_0,p)\right|\leq 1, \]
over $[0,1]\times\mathbb R^3\times\mathbb R^3$. Since any function $A\in\Psi^a$ is of order $O_A(t^{-5+a})$, the convolutions corresponding to $k>K$ contribute some term of order $O(t^{-3+K})$, and we have the expected estimate.
\end{proof}

At the expense of making things less explicit, we can eliminate the time variable in the expansion.

\begin{corollary}
\label{cor:tauexpansion}
There exist Schwartz functions $(\check u_n)_{n\geq0}$ such that for all $N\geq0$,
\[ u_{\tau^2}\big(p_0,P_{\tau,p_0}(\check p)\big)
 = \sum_{n=0}^N\tau^{-8+n}\,\check u_n\big(\check p\big) + O_N\big(\tau^{-7+N}\big), \]
where the quantifiers for the remainder are as in Corollary \ref{cor:additivebigO}. Moreover, $\check u_0=\check u$ and $\check u_1=0$.
\end{corollary}

The proof could also be extended to derivatives of the kernel, in the following way. For all indices $\alpha\in\mathbb N^3$, $m\geq0$, the difference
\[ \partial_\tau^mD^\alpha_{\check p}\Big(u_{\tau^2}\big(p_0,P_{\tau,p_0}(\check p)\big)\Big)
 - \sum_{n=0}^N\frac{(-8+n)!}{(-8+n-m)!}\tau^{-8+n-m}\cdot D^\alpha_{\check p}\check u_n\big(\check p\big) \]
is of order $O_N(\tau^{N-N_0})$, where $N_0$ depends on the order of differentiation. Note also that the derivative with respect to $p_0$ is zero.

\begin{proof}
Since the $k$-th convolution in the Duhamel series is in $\Psi^{k+1}$, there exists functions $f_k\in\check\Psi$ such that
\[ \big((E*)^k\widetilde u\big)_{\tau^2}\big(p_0,P_{\tau,p_0}(\check p)\big)
 = (\tau^2)^{-4+k}f_k(\tau,p_0,\check p). \]
Moreover, the whole problem and constructions are invariant under the action of isometries, so these functions are actually independent of $p_0$, and we can leave out this second variable. Defining the remainder $R_{k,N}$ through the relation
\[ \big((E*)^k\widetilde u\big)_{\tau^2}\big(p_0,P_{\tau,p_0}(\check p)\big)
 = \sum_{m=0}^{N-2k}\frac{\tau^{-8+2k+m}}{m!}(\partial_\tau^m f_k)(0,\check p) + R_{k,N}(\tau,\check p), \]
we get a uniform bound over for $\tau\leq 1$ in the form
\[ |R_{k,N}| \leq \frac{\tau^{-7+N}}{(N-2k+1)!}\big|(\partial_\tau^{N-2k+1} f_k)\big|_\infty. \]

Plugging these expansion in the estimate of Corollary \ref{cor:additivebigO} for some $K$ such that $2(-3+K) \geq -7+N$, we get an asymptotic expansion with
\[ \check u_n
 = \sum_{0\leq k\leq n/2}\frac1{(n-2k)!}(\partial_\tau^{n-2k} f_k)(0,\cdot). \]
Since the $f_k$ are in $\check\Psi$ and the sum is finite, the $\check u_n$ are of Schwartz class. From $f_0:(\tau,\check p)\mapsto\check u(\check p)$, we find $\check u_0=\check u$ and $\check u_1=0$.
\end{proof}

Since $\widetilde u$ and $u$ have different supports, it is impossible to have a fully multiplicative estimate of the form $u=\widetilde u\cdot(1+O(t))$. However, it holds within the support of $\widetilde u$.

\begin{corollary}
\label{cor:multiplicativebigO}
Let $F$ be a compact subset of $\{\check p:\check y^2<-2\check x\}$. Then
\[ u_t\big(p_0,p\big) = \tilde u_t\big(p_0,p\big)\big(1+O_F(t)\big) \]
for all $p_0\in\mathbb R^3$, $p\in P_{\sqrt t,p_0}F$. 

On the other hand, over the set of $(t,p_0,p)$ such that $p\in P_{\sqrt t,p_0}\{\check y^2\geq-2\check x\}$,
\[ u_t(p_0,p) = \widetilde u_t(p_0,p) = 0. \]
\end{corollary}

As an example, for $p_0=0$ and $K=0$, we get that we can write
\[ u_t\big(0,(t,0,0)+p\big) = \rho_t\big(0,(t,0,0)+p\big)\tilde u_t\big(0,(t,0,0)+p\big), \]
for all $0<t\leq 1$, $(x/t^2,y/t\sqrt t,\phi/\sqrt t)\in F$, where $\rho$ is a function asymptotically close to 1 in the sense that
\[ \sup_{0<t\leq1}\frac1t|\rho_t-1|_\infty<\infty. \]

\begin{proof}
The first part of the result is equivalent to its version in normalised coordinates, namely
\[ u_{\tau^2}\big(0,P_{\tau,0}(\check p)\big)
 = \tau^{-8}\check u(\check p)\big(1+O_F(\tau^2)\big) \]
for all $0<\tau\leq1$, $\check p\in F$.

From Corollary \ref{cor:tauexpansion}, we know that the difference between $u$ and the first order term is of order $\tau^{-6}$; specifically, there exist a universal constant $M$ such that
\[   \big|u_{\tau^2}\big(0,P_{\tau,0}(\check p)\big) - \tau^{-8}\check u(\check p)\big|
\leq M\tau^{-6} \]
for all $\tau\leq1$, $\check p\in\mathbb R^3$. Moreover, we know that $\{\check y^2<-2\check x\}$ is precisely the set of $\check p$ where $\check u$ is positive; see Proposition \ref{prop:strictsupport}. Since $F$ is compact, it means that $\check u$ is larger than some constant $\varepsilon>0$ over this set, so
\[ \big|u_{\tau^2}\big(0,P_{\tau,0}(\check p)\big) - \tau^{-8}\check u(\check p)\big|
\leq \frac M\varepsilon\tau^{-6}\check u(\check p) \]
there, which shows the expected estimate.

For the second part, since $P_{\sqrt t,p_0}\{\check y^2\geq-2\check x\}$ is precisely the set where $\widetilde u_t(p_0,\cdot)$ is zero, see again Proposition \ref{prop:strictsupport}, it remains to show that $u$ is also zero over this set. In other words, we want to show that for all $\tau$,
\[ \check p\mapsto u_{\tau^2}\big(0,P_{\tau,0}(\check p)\big) \]
is zero whenever $\check p$ satisfies $\check y^2\geq-2\check x$. Because $t\mapsto(x_t,y_t)$ moves at velocity 1, the density is zero at $p$ for $x^2+y^2>t^2$, and by continuity this also holds when the inequality is large. In the rescaled coordinates, this condition is exactly
\[ (\tau^2+\tau^4\check x)^2 + (\tau^3\check y)^2 \geq \tau^4
\Leftrightarrow
   \check y^2 + 2\check x\geq-\tau^2\check x^2. \]
It clearly holds over $\{\check y^2\geq-2\check x\}$.
\end{proof}

\section{The macroscopic and mesoscopic scales}
\label{sec:LDP}

In the above, we have been considering the limit $p$ close to $p_0$ as $t$ goes to zero, in the scale described by the dilation operators $P_{\sqrt t,p_0}$. Let us call this scale the microscopic scale. In this section, we present large deviation principles in the macroscopic and mesoscopic scale, for $t$ small. For simplicity, throughout this section we fix $p_0=0$; the case of a general $p_0$ follows from the invariance of the problem under isometries of $\mathbb R^2$. The results are stated in $\mathbb R^3$, but because of the contraction principle they translate directly to the more geometric $\mathbb R^2\times\mathbb S^1$, where the angle $\phi$ is considered up to $2\pi\mathbb Z$.

In the macroscopic scale, it is clear that any point $p$ cannot be reached in time $t<|(x,y)|$. This means that the kernel $u_t(0,p)$ is eventually zero for all points $p$ that do not lie on the line $\{(0,0,\phi),\phi\in\mathbb R^3\}$. On this line, the process is fairly close to Gaussian, as we can see in the following statement.

\begin{proposition}
The variables $(p_t)_{t>0}$ satisfy a large deviation principle with rate function
\[ I_\mathrm{macro}:p = (x,y,\phi) \mapsto
   \begin{cases}
     \frac{|\phi|^2}2&\text{for }(x,y)=0, \\
     \infty&\text{else.}
   \end{cases} \]
Namely, for any Borel subset $S\subset\mathbb R^3$,
\begin{multline*}
     -\inf_{p\in\operatorname{int}(S)}I_\mathrm{macro}(p)
\leq \liminf_{t\to0}t\ln\mathbb P\big(p_t\in\operatorname{int}(S)\big) \\
\leq \limsup_{t\to0}t\ln\mathbb P\big(p_t\in\overline S\big)
\leq -\inf_{p\in\overline S}I_\mathrm{macro}(p).
\end{multline*}
\end{proposition}

\begin{proof}
We show the result from the following two intermediate statements.
\begin{itemize}
\item For any open box $R=(a_x,b_x)\times(a_y,b_y)\times(a_\phi,b_\phi)$, the large deviation principle holds as an equality:
\[ \lim_{t\to0}t\ln\big(\mathbb P(p_t\in R)\big) = -\inf_{p\in R}I(p). \]
\item The occupation probability of complements of compact sets decrease exponentially, with arbitrarily large rate:
\[ \lim_{M\to\infty}\limsup_{t\to0}t\ln\mathbb P\big(p_t\notin[-M,M]^3\big)=-\infty. \]
\end{itemize}
The proposition follows because the lower bound is local, and the upper bound is stable under finite covers.

The second fact is a consequence of the fact that the probability of escape decreases in a Gaussian fashion; this is totally analogous to the proof of Proposition \ref{prop:exponentialFeller}. As for the first one, fix such a box $R$. If $0$ does not belong to $(a_x,b_x)\times(a_y,b_y)$, the probability is eventually zero, and the result is obvious. Otherwise, this rectangle actually contains a small open disc $\{x^2+y^2<r^2\}$. Then for all $t<r$,
\[ \mathbb P(p_t\in R)
 = \mathbb P\big(\phi_t\in(a_\phi,b_\phi)\big)
 = \mathbb P\big(W_t\in(a_\phi,b_\phi)\big) \]
and the result follows from the large deviation principle for Gaussian variables:
\[ \lim_{t\to0}t\ln\big(\mathbb P(p_t\in R)\big)
 = \lim_{t\to0}t\ln\mathbb P\big(W_t\in(a_\phi,b_\phi)\big)
 = -\inf_{w\in[a_\phi,b_\phi]}\frac{|w|^2}2
 = -\inf_{p\in R}I(p). \qedhere \]
\end{proof}

This result shows that any kind of interesting large deviation principle involving $x$ and $y$ should take place at a different scale. Since the problem was that we could not reach $(x,y,\phi)$ in time smaller than $|(x,y)|$, we look at the variable $(x_t/t,y_t/t,\phi_t)$ (the mesoscopic scale), so that every point with in the cylinder $\{(x,y,\phi):x^2+y^2<1\}$ can be reached for any given time. This time, we get a rate function depending non-trivially in $(x,y)$.

Let $I_W$ be the celebrated rate function for Brownian motion, i.e. the energy defined on continuous curves $w\in\mathcal C([0,1])$ by
\[ I_W:w\mapsto
   \begin{cases}
   \frac12\int_0^1|\dot w(u)|^2\mathrm du & \text{if }w:u\mapsto \int_0^u\dot w(r)\mathrm dr\text{ for some }\dot w\in L^2([0,1]), \\
   \infty & \text{else.}
   \end{cases} \]
In the following, we will simply write $\frac12\int_0^1|w'(t)|^2\mathrm dt$, with the understanding that the quantity is infinite if a representation of $w$ as the integral of a square integrable derivative does not exist. Schilder's theorem asserts that the collection of variables $(\sqrt tW)_{t>0}$ satisfy a large deviation principle in $\mathcal C([0,1])$ with rate function $I_W$.

\begin{proposition}
\label{prop:LDPmeso}
For a fixed $p=(x,y,\phi)\in\mathbb R^3$, we set $\Gamma(p)$ the set of continuous curves $\gamma:[0,1]\to\mathbb R$ such that
\begin{align*}
\gamma(0)&=0, & \gamma(1)&=\phi, \\
\int_0^t\cos(\gamma(t))\mathrm dt&=x, & \int_0^t\sin(\gamma(t))\mathrm dt&=y.
\end{align*}
For instance, $\Gamma(p)=\varnothing$ for $|(x,y)|>1$. Then the variables $(x_t/t,y_t/t,\phi_t)_{t>0}$ satisfy a large deviation principle with rate function
\[ I_\mathrm{meso}:p\mapsto
   \inf_{\gamma\in\Gamma(p)}I_W(\gamma)
 = \frac12\inf_{\gamma\in\Gamma(p)}\int_0^1|\gamma'(t)|^2\mathrm dt. \]
\end{proposition}

Although we believe this result to be interesting in itself, a concrete description of the problem would require a good understanding of $I_\mathrm{meso}$. This function is actually difficult to estimate; see \cite[sec. IV.2]{PThesis} for an attempt at the study of this functional, using variational techniques. Here are a few partial results using elementary methods.

\begin{proposition}
\label{prop:Imeso}~

\begin{enumerate}
\item The rate function is extremal along the following sets:
\begin{align*}
\{I_\mathrm{meso}=0\} &= \{(1,0,0)\}, &
\{I_\mathrm{meso}=\infty\} &= \{p\neq(1,0,0)\text{ such that }|(x,y)|\geq1\}.
\end{align*}
\item If $p$ is a point reached along a path of constant curvature, i.e.
\begin{align*}
x &= \int_0^1\cos(Ct)\mathrm dt,  &  y &= \int_0^1\sin(Ct)\mathrm dt,  &  \phi &= C,
\end{align*}
then $I_\mathrm{meso}(p)=C^2/2$. For instance, $I_\mathrm{meso}(0,0,\pm2\pi)=2\pi^2$ for $C=\pm2\pi$.
\item The above path is the easiest way to get back to the origin:
\[ \inf_{k\in\mathbb Z,|k|\neq1}I_\mathrm{meso}(0,0,2k\pi)>2\pi^2. \]
In particular, the corresponding rate function in $\mathbb R^2\times\mathbb S^1$ is $2\pi^2$ at the origin.
\end{enumerate}
\end{proposition}

\begin{remark}
This is very different to the sub-elliptic case. In the latter, the diffusion satisfies a large deviation principle with rate function $p\mapsto-d(p_0,p)^2$, for $d$ the so-called Carnot-Carathéodory distance (before renormalisation) --- see for instance the results of Léandre \cite{LeandreMaj,LeandreMin}. In particular, the kernel concentrates on the diagonal. In our case, the diagonal is far too distant to concentrate the mass: starting from $0$, the probability that $x_t$ be at most $t/2$ is easily seen to be exponentially small. If we think that $p_0+tX(p_0)$ plays the role of the diagonal in our case, it is true that the mass concentrate \emph{around} there, but it does not quite concentrates \emph{at} this point. In fact, since the support of $u_t(0,\cdot)$ is included in $\{p:x^2+y^2\leq t^2\}$, the value of the kernel at this point is zero. In the sub-Riemannian case, the kernel on the diagonal behaves as a negative power of $t$, in a similar way as $t^{-4}$ is our dominant term. This is a tiny fragment of a celebrated result of Ben Arous, see \cite{BenArous}.
\end{remark}

\begin{proof}[Proof of Proposition \ref{prop:LDPmeso}]
Let us fix $t>0$, and define the Brownian motion $\widetilde W:u\mapsto t^{-1/2}W_{tu}$. Then
\begin{align*}
\frac{x_t}t &= \int_0^1\cos\big(\sqrt t\widetilde W_u\big)\mathrm du, &
\frac{y_t}t &= \int_0^1\sin\big(\sqrt t\widetilde W_u\big)\mathrm du, &
\phi_t &= \sqrt t\widetilde W_1.
\end{align*}
This means that the vector depends continuously on $\sqrt t\widetilde W$, seen as a vector in $\mathcal C([0,1])$ with the uniform topology, say according to some function $F:\mathcal C([0,1])\to\mathbb R^3$. Since the latter satisfies a large deviation principle with rate function $I_W$, the contraction principle implies that $(x_t/t,y_t/t,\phi_t)$ satisfies a large deviation principle with rate function
\[ I_\mathrm{meso}:p\mapsto\inf_{w:F(w)=p}I_W(w). \]

The set $F^{-1}\{p\}$ is exactly the $\Gamma(p)$ of the proposition, and we get the expected expression for $I_\mathrm{meso}$.
\end{proof}

\begin{proof}[Proof of Proposition \ref{prop:Imeso}]
We use the same function $F$ as in the previous proof, sending a continuous curve $\gamma\in\mathcal C([0,1])$ to the endpoint $p$ with
\begin{align*}
x &= \int_0^1\cos\big(\gamma(u)\big)\mathrm du, &
y &= \int_0^1\sin\big(\gamma(u)\big)\mathrm du, &
\phi &= \gamma(1).
\end{align*}

\begin{enumerate}
\item Clearly $\Gamma(1,0,0)=\{0\}$ and $\Gamma(p)=\varnothing$ for $|(x,y)|\geq1$ with $p\neq(1,0,0)$. We will have $\Gamma(p)\neq\varnothing$ if there exists a smooth curve of length 1 in the plane joining $0$ and $(x,y)$, with non-vanishing first derivative and prescribed starting and ending angle (note that the angle belongs to $\mathbb R$ rather than $\mathbb S^1$, so one might need to spin a few times). The existence is clear on a drawing for any $p$ such that $|(x,y)|<1$.

It remains only to show that $I_\mathrm{meso}(p)=0$ happens only for $p=(1,0,0)$. If $p$ is such a point, then there exists a sequence of curves $\gamma_n$ with $F(\gamma_n)=p$ and energy tending to zero. Using the Cauchy-Schwarz inequality, it means that $(\gamma_n)_{n\geq0}$ tends to zero in the continuous topology, so $p=F(\gamma_n)$ tends to $F(0)$ as $n$ goes to infinity, and $p=F(0)=(1,0,0)$.

\item Fix $C$, and consider the point $p=F(t\mapsto Ct)$. It is clear, considering $\gamma:t\mapsto Ct$, that $I(p)\leq C^2/2$. Going the other direction, for a curve $\gamma\in\Gamma(p)$,
\[   \frac12\int_0^1\gamma'(t)^2\mathrm dt
\geq \frac12\left(\int_0^1\gamma'(t)\mathrm dt\right)^2
   = \frac12\gamma(1)^2 = \frac{\phi^2}2 = C^2/2. \]

\item Considering the case of constant curvature with $C=2k\pi$, $k\neq0$,
\[ \inf_{k\in\mathbb Z,|k|>1}I(0,0,2k\pi) = \inf_{k\geq2}\frac{(2k\pi)^2}2 > 2\pi^2. \]
It remains to show that $I(0,0,0)>2\pi^2$. Fix $p=0$ and $\gamma\in\Gamma(0)$. Since it is continuous, the image of $\gamma$ in $\mathbb R$ is a closed interval $[a,b]$ containing zero. Suppose $b-a\leq\pi$; then
\begin{align*}
\Re\left(\mathbf e^{-\mathsf i(a+b)/2}(x+\mathsf iy)\right)
   &= \Re\int_0^1\mathbf e^{-\mathsf i(a+b)/2+\mathsf i\gamma(t)}\mathrm dt \\
   &= \int_0^1\cos\Big(\gamma(t)-\frac{a+b}2\Big)\mathrm dt \\
   &> \int_0^1\cos\Big(\frac{b-a}2\Big)\mathrm dt \geq 0.
\end{align*}
The strict inequality comes from the fact that $\gamma$ cannot be constant. This is not possible, since $p=0$ and the left hand side vanishes; hence $b-a>\pi$. Let us call $0\leq t^-<1$ the first time $\gamma$ hits $\{a,b\}$, and $0<t^+\leq 1$ the first time it hits the second member of the pair. Up to considering $-\gamma$, we assume without loss of generality that $\gamma(t^-)=a$.

Decomposing the energy of the curve, we get
\begin{align*}
      \frac12\int_0^1|\gamma'(t)|^2\mathrm dt
&\geq \frac12\left(\int_0^1|\gamma'(t)|\mathrm dt\right)^2 \\
&\geq \frac12\left( \left|\int_0^{t^-}\gamma'(t)\mathrm dt\right|
                  + \left|\int_{t^-}^{t^+}\gamma'(t)\mathrm dt\right|
                  + \left|\int_{t^+}^1\gamma'(t)\mathrm dt\right|\right)^2 \\
   &= \frac12(-a + (b-a) + b)^2 \\
&\geq \frac12(2\pi)^2
    = 2\pi^2.
\end{align*}
Since this lower bound holds for all $\gamma\in\Gamma(p)$, this implies $I(0)\geq2\pi^2$.

This is enough to get $I^\mathbb{S}(0)=2\pi^2$, but we can actually show the inequality is strict. Indeed, if it is not, then we can find a sequence of curves $\gamma_n\in\Gamma(p)$ whose energy decrease to $2\pi^2$. Define $a_n\leq0\leq b_n$, $0\leq t^-_n<t^+_n\leq 1$ as above with respect to the curve $\gamma_n$. Up to replacing a number of $\gamma_n$ by $-\gamma_n$, we can assume that $\gamma_n(t^-_n)=a_n$ for all $n$. Define
\[ \overline\gamma_n:t\mapsto\int_0^t\left( \frac{a_n}{t^-_n}\mathbf1_{[0,t^-_n]}(u)
                                          + \frac{b_n-a_n}{t^+_n-t^-_n}\mathbf1_{[t^-_n,t^+_n]}(u)
                                          - \frac{b_n}{1-t^+_n}\mathbf1_{[t^+_n,1]}(u) \right)
                                     \mathrm du. \]
We introduce these functions because
\[ \int_0^1\big(\gamma_n'(t)-\overline\gamma_n'(t)\big)\overline\gamma_n'(t)\mathrm dt = 0; \]
in some sense that can be made precise, $\overline\gamma_n$ is a conditional expectation of $\gamma_n$. In particular,
\[ \int_0^1|\gamma_n'(t)|^2\mathrm dt
 = \frac12\int_0^1|\gamma_n'(t)-\overline\gamma_n'(t)|^2\mathrm dt
 + \frac12\int_0^1|\overline\gamma_n'(t)|^2\mathrm dt. \]

Using the same argument as above,
\begin{align*}
      \frac12\int_0^1|\gamma_n'(t)|^2\mathrm dt
   &= \frac12\int_0^1|\gamma_n'(t)-\overline\gamma_n'(t)|^2\mathrm dt
    + \frac12\int_0^1|\overline\gamma_n'(t)|^2\mathrm dt \\
&\geq \frac12\sup_{0<t\leq1}\frac1t|\gamma_n(t)-\overline\gamma_n(t)| + \frac12\big(2(b_n-a_n)\big)^2
\end{align*}
with $b_n-a_n\geq\pi$. Since the left hand side converges to $2\pi^2$, it means that $b_n-a_n$ converges to $\pi$, and that $\gamma_n$ becomes uniformly close to $\overline\gamma_n(t)$. In particular, we can extract a subsequence such that $a_n$, $b_n$, $t^-_n$ and $t^+_n$ converge to $a,b,t^-,t^+$, and along this subsequence $\gamma_n$ converges uniformly to
\[ \overline\gamma:t\mapsto\int_0^t\left( \frac{a}{t^-}\mathbf1_{[0,t^-]}(u)
                                        + \frac{b-a}{t^+-t^-}\mathbf1_{[t^-,t^+]}(u)
                                        - \frac{b}{1-t^+}\mathbf1_{[t^+,1]}(u) \right)
                                   \mathrm du. \]
By continuity of $F$, this means that $F(\overline\gamma)=p$. But we have seen that a curve satisfying this condition must satisfy $b-a>\pi$: we have a contradiction.\qedhere
\end{enumerate}
\end{proof}

Based on the Hamiltonian approach to the problem, we conjecture that the optimal curves for $p=0$ are the solution to the equation
\begin{align*} \gamma''(t) &= -a\sin\big(\gamma(t)\big), & \gamma(0)&=0,\ \gamma'(0)=b, \end{align*}
together with its translations. Here, $a,b>0$ are the only parameters such that $\gamma$ is 1-periodic and the integral of $\cos(\gamma)$ is zero over this interval (equivalently, $\frac1T\int_0^T\cos(\gamma(s))\mathrm ds$ converges to zero as $T$ goes to infinity).

\appendix

\section{Qualitative regularity for the diffusion}
\label{app:Fellerandco}

In this section we show that the diffusions $p$ and $\widetilde p$ are Feller and hypoelliptic.

\subsection{Hypoellipticity}
\label{ssec:hypoelliptic}
The diffusion will be hypoelliptic if it satisfies the so-called Hörmander condition: every vector field can be written as a linear combination of $\Phi$, and Lie brackets involving $\Phi$ and $X$, with smooth functions as coefficients. See \cite{Hormander} for the original proof of Hörmander. Unfortunately, the original article \cite{Malliavin} presenting Malliavin's alternative proof, for the purpose of which he introduced his calculus, is difficult to get by.

It is easily seen that
\begin{align*}
[X,Y]&=0, & [X,\Phi] &= -Y, & [Y,\Phi] &= X.
\end{align*}
This proves that $\Phi$, $[X,\Phi]$ and $[[X,\Phi],\Phi]$ generate the tangent space at every point, so the original diffusion is regular. As for the approximate one,
\begin{align*}
[\widetilde X,\Phi]&=(\phi-\phi_0)X(p_0)-Y(p_0), & [-(\phi-\phi_0)X(p_0)+Y(p_0),\Phi]&=X(p_0),
\end{align*}
so it is again hypoelliptic.

\subsection{Feller properties}
Looking at the defining equations for $p$ and $\widetilde p$, we can easily show that they are defined for all times. For instance, for $p_0=0$, we have
\begin{align*}
x_t&=\int_0^t\cos(W_s)\mathrm ds, & y_t&=\int_0^t\sin(W_s)\mathrm ds, & \phi_t&=W_t,
\end{align*}
which does not blow up since $W$ is continuous almost surely. For $\widetilde p$, we have in some sense two initial points: the point of reference $p_0$ where we expand $p\mapsto X(p)$, and the initial condition $\widetilde p_0$. In the case $p_0=0$, we have
\begin{align*}
\widetilde x_t&=\widetilde x_0+\int_0^t\Big(1-\frac{(\widetilde \phi_0+W_s)^2}2\Big)\mathrm ds, &
\widetilde y_t&=\widetilde y_0+\int_0^t(\widetilde \phi_0+W_s)\mathrm ds, &
\widetilde\phi_t&=\widetilde\phi_0+W_t,
\end{align*}
so it is defined for all times for the same reason. The case $p_0\neq0$ is similar; in fact it is exactly the same up to isometries of the underlying plane.

As solutions of complete stochastic differential equations with smooth coefficients, they have semigroup mapping the space of continuous bounded functions to itself (this can be seen as a consequence of the fact that solutions of such equations with uniformly Lipschitz coefficients are Feller, see for instance Theorem 2.5 in \cite{RevuzYor}, with a localisation argument). To see that they are actually Feller, it suffices to show that for any fixed compact set $K$, the function $p_0\mapsto\mathbb P_{p_0}(p_t\in K)$ (resp. $\widetilde p_0\mapsto\mathbb P_{\widetilde p_0}(\widetilde p_t\in K)$) vanishes at infinity. This is the object of the following result.

\begin{proposition}
\label{prop:exponentialFeller}
Let $K$ be a compact subset of $\mathbb R^3$, and fix $t>0$. Then
\begin{align*}
p_0&\mapsto\mathbb P_{p_0}(p_t\in K), &
\text{resp. }\widetilde p_0&\mapsto\mathbb P_{\widetilde p_0}(\widetilde p_t\in K),
\end{align*}
exhibit Gaussian decay, resp. exponential decay.

In particular, their semigroups map compactly supported continuous functions to integrable ones.
\end{proposition}

\begin{proof}
It is enough to show that for all $M>0$, $p_0\mapsto\mathbb P_{p_0}(|p_t|\leq M)$ decreases as stated, and similarly for $\widetilde p$.

Let us fix such $M$. For the diffusion $p$, the velocity of $(x,y)$ is one, so we know that the function is zero for $(x_0,y_0)$ large enough. This means that we can restrict to $(x_0,y_0)$ in a compact set. Moreover, for the same reason, we have
\[   \mathbb P_{p_0}(|p_t|\leq M)
\leq \mathbb P_{p_0}(|\phi_t|\leq M)
\leq \mathbb P_{p_0}(|\phi_t-\phi_0|\geq |\phi_0|-M). \]
Because $\phi_t-\phi_0$ is Gaussian (centred with variance $t$), this decreases as the exponential of a square as a function of $|\phi_0|$. Combined with the case $(x_0,y_0)$ outside of a compact set, we get the expected result, say in the following form: there exists $c>0$ (depending on $t$ and $M$) such that
\[ \mathbb P_{p_0}(|p_t|\leq M)\leq2\exp(-c|p_0|^2). \]

For $\widetilde p$, we work with $p_0=0$. As noted above, this amounts to acting via an isometry on the underlying plane, so we can do this without loss of generality. We then use the explicit expression for $\widetilde p$ once again:
\begin{align*}
\widetilde x_t&=\widetilde x_0+\int_0^t\Big(1-\frac{(\widetilde \phi_0+W_s)^2}2\Big)\mathrm ds, &
\widetilde y_t&=\widetilde y_0+\int_0^t(\widetilde \phi_0+W_s)\mathrm ds, &
\widetilde\phi_t&=\widetilde\phi_0+W_t,
\end{align*}
in particular
\begin{align*}
      |\widetilde x_t-\widetilde x_0|
&\leq t+t|\widetilde \phi_0|^2+t\sup_{s\leq t}|W_s|^2, \\
      |\widetilde y_t-\widetilde y_0|
&\leq t|\widetilde \phi_0|+t\sup_{s\leq t}|W_s|, \\
      |\widetilde\phi_t-\widetilde\phi_0|
&\leq \sup_{s\leq t}|W_s|.
\end{align*}
One can deduce that there exists a constant $C>0$ large enough (depending on $t$) such that
\[ |\widetilde p_t-\widetilde p_0|\leq C\big(1+\sup_{s\leq t}|W_s|^2+|\widetilde\phi_t|^2\big). \]
Coming back to the probability under consideration,
\begin{align*}
      \mathbb P_{\widetilde p_0}(|\widetilde p_t|\leq M)
&\leq \mathbb P_{\widetilde p_0}\Big(|\widetilde p_0|-M\leq C\big(1+\sup_{s\leq t}|W_s|^2+M^2\big)\Big) \\
&\leq \mathbb P_{\widetilde p_0}\big(C'\sup_{s\leq t}|W_s|^2\geq|\widetilde p_0|-C'\big)
\end{align*}
for $C'>0$ another large constant. Since $\sup_{s\leq t}|W_s|$ has Gaussian tails, we deduce that this probability must decrease exponentially, say in the sense that there exists $c=c(t,M)>0$ such that
\[ \mathbb P_{\widetilde p_0}(|\widetilde p_t|\leq M)\leq2\exp(-c|\widetilde p_0|).\qedhere \]
\end{proof}

\subsection{Time reversal}
\label{ssec:timereversal}

In section \ref{sec:FokkerPlanck}, we used the dual diffusions $p^*$, whose generator is dual to that of $p$. Because the vector fields $X$ and $\Phi$ are divergence-free, its definition in terms of stochastic differential equations are
\begin{align*}
\mathrm dp^*_t &= -X(p^*_t)\mathsf dt+\Phi(p^*_t)\mathrm dW_t, &
\mathrm d\widetilde p^*_t &= -\widetilde X(\widetilde p^*_t)\mathrm dt+\Phi(\widetilde p^*_t)\mathrm dW_t.
\end{align*}
It looks like this negative sign should not change the behaviour of the diffusion much, and it is indeed the case: every result shown for $p$ in this section also holds for $p^*$.

We have also stated in Lemma \ref{lem:timereversal} that the kernels of $p$ and $p^*$ are in a sense dual. Given the results already proved in this section, it only remains to show the last property. Because the vector fields $X$ and $\Phi$ are divergence-free, this is an immediate consequence of Theorem 3.50 in \cite{StroockManifolds}; see also Theorem 4.17 for a version on manifolds.

\subsection{Proof of Lemma \ref{lem:convolutionwu}}
\label{ssec:convolutionwu}

Recall that we want to show that for $f:\mathbb R^3\to\mathbb R$, $f\star\widetilde u_t$ is bounded by $|f|_\infty$ for $f$ bounded, and converges uniformly to $f$ of $f$ is also Lipschitz.

\begin{proof}[Proof of Lemma \ref{lem:convolutionwu}]
Let $f:\mathbb R^3\to\mathbb R$ be bounded. For $p\in\mathbb R^3$ fixed, the space convolution is
\[ (f\star\widetilde u_t)(p) = \int_{\mathbb R^3}f(p_*)\widetilde u_t(p_*,p)\mathrm dp_*. \]
We make the following implicit change of variables: $p_*$ is seen as a function of $\check p_*$, through the relation
\[ p = P_{\sqrt t,p_*}(\check p_*). \]
Explicitly, setting $q=(0,0,\phi-\sqrt t\check\phi)$,
\[ p_* = p_*(\check p) = p - tX(q) - M_qT_{\sqrt t}\,\check p. \]
Since the Jacobian of this transformation is $t^4$, we can use the representation of $u$ in terms of $\check u$ to see that
\[ (f\star\widetilde u_t)(p)
 = \int_{\mathbb R^3}f\big(p_*(\check p_*)\big)\check u(\check p_*)\mathrm d\check p_*, \]
which gives the immediate bound
\[   |f\star\widetilde u_t|_\infty
\leq |f|_\infty\cdot|\check u|_{\mathrm L^1}
   = |f|_\infty. \]

Suppose moreover that $f$ is Lipschitz. Then
\begin{align*}
      |(f\star\widetilde u_t)(p) - f(p)|
&\leq \int_{\mathbb R^3}\big|f\big(p_*(\check p_*)\big)-f(p)\big|\check u(\check p_*)\mathrm d\check p_* \\
&\leq |f|_{\mathrm{Lip}}\int_{\mathbb R^3}\big|p_*(\check p_*)-p\big|\check u(\check p_*)\mathrm d\check p_* \\
&\leq \sqrt t\cdot(1+\sqrt t)|f|_{\mathrm{Lip}}\int_{\mathbb R^3}(1+|\check p_*|)\check u(\check p_*)\mathrm d\check p_*.
\end{align*}
This last expression does not depend on $p$. Since this last integral is bounded (it is the expectation of $1+|\xi|$, where $\xi$ is defined in equation \eqref{eq:defxi}, hence it is stochastically bounded by a quadratic polynomial in a Gaussian), we get the expected convergence in the uniform topology.
\end{proof}

\section{Closed form for the Fourier transform}
\label{app:explicitFourier}

Here we look for an explicit expression for the characteristic function of $\xi$, as defined in equation~\eqref{eq:defxi}. As discussed in the beginning of section \ref{sec:checku}, it is also the Fourier transform of the approximated kernel $\widetilde u$, up to a change of scale.

It will be useful to introduce the \emph{ad hoc} cosine and sine functions
\begin{align*}
\operatorname{cah}:\zeta&\mapsto\sum_{k\geq0}\frac{\zeta^k}{(2k)!} &
& \text{and} &
\operatorname{sah}:\zeta&\mapsto\sum_{k\geq0}\frac{\zeta^k}{(2k+1)!}.
\end{align*}
They are entire functions, and for any $\alpha\in\mathbb C^*$,
\begin{align*}
\operatorname{cah}\big(\alpha^2\big)&=\cosh(\alpha) &
& \text{and} &
\operatorname{sah}\big(\alpha^2\big)&=\sinh(\alpha)/\alpha.
\end{align*}
In particular one can check the important property that they have no zero over $\{\zeta:\Re\zeta>-\pi^2/4\}$. Namely, their zeros are of the form $-\pi^2k^2/4$, with $k$ odd for $\operatorname{cah}$ and even non-zero for $\operatorname{sah}$. Over this domain, we can therefore define $\operatorname{tah}:=\operatorname{sah}/\operatorname{cah}$.

\begin{theorem}
The characteristic function of the variable $\xi$ defined in equation \eqref{eq:defxi}, with density $\check u$, is given by
\begin{multline}
\label{eq:explicitFourier}
(\lambda,\mu,\nu)\mapsto\mathbb E\left[\mathsf e^{\mathsf i(\lambda,\mu,\nu)\cdot\xi}\right] \\
= \frac1{\sqrt{\operatorname{cah}(\mathsf i\lambda)}}
\exp\left(
- \frac 1{2\mathsf i\lambda}\Big(1-\operatorname{tah}(\mathsf i\lambda)\Big)\mu^2
- \frac1{\mathsf i\lambda}\Big(1-\frac1{\operatorname{cah}(\mathsf i\lambda)}\Big)\mu\nu
- \frac12\operatorname{tah}(\mathsf i\lambda)\nu^2
\right),
\end{multline}
where $\zeta\mapsto\sqrt{\operatorname{cah}(\zeta)}$ is the unique analytic continuation over $\{\zeta:\Re\zeta>-\pi^2/4\}$ of the obvious function over $(0,\infty)$.
\end{theorem}

To get this expression, we use the decomposition of the Brownian motion $W$ in Fourier series, and express $\xi$ as the sum of the associated $\xi_n$, as defined in equation \eqref{eq:defxin}. They are quadratic polynomials in Gaussian variables, so their characteristic functions are explicit, namely they were given in equation \eqref{eq:Fourierxin}.

We can rephrase the expression for the characteristic function of $\xi_n$ as
\begin{gather*}
   \mathbb E\left[\mathsf e^{\mathsf i(\lambda,\mu,\nu)\cdot\xi_n}\right]
 = I_{\lambda,n}
   \exp\big( - Q_{\lambda,n}(\mu,\nu)\big), \\
\begin{aligned}
I_{\lambda,n} &:= {\left(1+\frac{\mathsf i\lambda}{\pi^2(n+1/2)^2}\right),}^{\hspace{-5pt}-1/2} &
Q_{\lambda,n}(\mu,\nu) &:= \frac{\big(\frac{(-1)^n\mu}{\pi(n+1/2)}+\nu\big)^2}{\pi^2(n+1/2)^2+\mathsf i\lambda}.
\end{aligned}
\end{gather*}
Here $\zeta\mapsto\zeta^{-1/2}$ is the analytic continuation of the usual convention on $\mathbb R^*_+$ to the half plane $\{\zeta:\Re\zeta>0\}$. As previously discussed, the characteristic function for $\xi$ is given by the convergent product
\[ \mathbb E\left[\mathsf e^{\mathsf i(\lambda,\mu,\nu)\cdot\xi}\right]
 = \prod_{n\geq0}I_{\lambda,n}\cdot\exp\left(-\sum_{n\geq0}Q_{\lambda,n}(\mu,\nu)\right)
 =: I_\lambda\exp\big(-Q_\lambda(\mu,\nu)\big). \]

Using the infinite product decomposition
\[ \cos(\theta) = \prod_{n\geq0}\left(1-\frac{\theta^2}{(n+1/2)^2\pi^2}\right) \]
for the cosine \cite[§4.3.90]{AbramowitzStegun}, we see that
\begin{equation} \label{eq:sqrtcah}
   \bigg(\prod_{n\geq0}I_{\lambda,n}\bigg)^{-2}
 = \prod_{n\geq0}\left(1+\frac{\mathsf i\lambda}{\pi^2(n+1/2)^2}\right)
 = \operatorname{cah}(\mathsf i\lambda).
\end{equation}
In particular, setting $\zeta\mapsto\sqrt{\operatorname{cah}(\zeta)}$ as the analytic function defined over $\{\zeta:\Re\zeta>-\pi^2/4\}$ that coincides with the classical square root of $\operatorname{cah}$ over $\mathbb R^*_+$, we get by continuity that
\[ I_{\lambda}=\prod_{n\geq0}I_{\lambda,n} = \frac1{\sqrt{\operatorname{cah}(\mathsf i\lambda)}}. \]

For the coefficient of $Q_{\lambda,n}$ corresponding to $\nu^2$, we use the infinite partial fraction decomposition for the cotangent
\[ \tan(\theta) = 2\theta\sum_{n\geq0}\frac1{\pi^2(n+1/2)^2-\theta^2}. \]
This relation can be deduced from the representation of the cosine as the above infinite product, by noticing $\tan$ is the derivative of $\theta\mapsto-\ln\cos(\theta)$. This gives
\[ Q_\lambda(0,1) = \sum_{n\geq0}\frac1{\pi^2(n+1/2)^2+\mathsf i\lambda}
                  = \frac12\operatorname{tah}(\mathsf i\lambda). \]

Regarding the coefficient $\mu^2$, we use the same expansion for the tangent, and the value of the sum
\[ \zeta(2) = \sum_{n>0}\frac1{n^2} = \frac{\pi^2}6. \]
Since we have
\begin{multline*}
   \sum_{n\geq0}\frac1{\pi^2(n+1/2)^2(\pi^2(n+1/2)^2-\theta^2)} \\
\begin{aligned}
&= \frac1{\theta^2}\sum_{n\geq0}\frac1{\pi^2(n+1/2)^2-\theta^2}
 - \frac1{\theta^2}\sum_{n\geq0}\frac1{\pi^2(n+1/2)^2} \\
&= \frac1{\theta^2}\cdot\frac{\tan(\theta)}{2\theta}
 - \frac1{\pi^2\theta^2}\left(\sum_{n>0}\frac1{(n/2)^2} - \sum_{n>0}\frac1{n^2}\right) \\
&= -\frac1{2\theta^2}\Big(1-\frac{\tan(\theta)}\theta\Big),
\end{aligned}
\end{multline*}
we deduce that
\[ Q_\lambda(1,0) = \sum_{n\geq0}\frac1{\pi^2(n+1/2)^2(\pi^2(n+1/2)^2+\mathsf i\lambda)}
                  = \frac1{2\mathsf i\lambda}\big(1-\operatorname{tah}(\mathsf i\lambda)\big). \]

For the last coefficient of $Q_\lambda$, we use the representation
\[ \frac1{\cos(\theta)} = 2\sum_{n\geq0}\frac{(-1)^n\pi(n+1/2)}{\pi^2(n+1/2)^2-\theta^2}. \]
for the secant, which can be deduced from $\cos(\theta)=\sin(\pi/2+\theta)$ and \cite[§4.3.93]{AbramowitzStegun} under the form
\[ \frac1{\sin(\theta)}
 = \frac1\theta+2\theta\sum_{n\geq1}\frac{(-1)^n}{\theta^2-n^2\pi^2}
 = \sum_{n\geq0}\frac{(-1)^n}{\theta+n\pi}-\sum_{n\geq0}\frac{(-1)^n}{\theta-(n+1)\pi}. \]

We can then compute the sum
\begin{multline*}
  \sum_{n\geq0}\frac{(-1)^n}{\pi(n+1/2)(\pi^2(n+1/2)^2-\theta^2)} \\
\begin{aligned}
&= \frac1{\theta^2}\sum_{n\geq0}\frac{(-1)^n\pi(n+1/2)}{\pi^2(n+1/2)^2-\theta^2}
 - \frac1{\theta^2}\sum_{n\geq0}\frac{(-1)^n\pi(n+1/2)}{\pi^2(n+1/2)^2-0^2} \\
&= -\frac1{2\theta^2}\left(1-\frac1{\cos(\theta)}\right),
\end{aligned}
\end{multline*}
from which we get
\begin{align*}
   Q_\lambda(1,1)-Q_\lambda(1,0)-Q_\lambda(0,1)
&= 2\sum_{n\geq0}\frac{(-1)^n}{\pi(n+1/2)(\pi^2(n+1/2)^2+\mathsf i\lambda)} \\
&= \frac1{\mathsf i\lambda}\left(1-\frac1{\operatorname{cah}(\mathsf i\lambda)}\right).
\end{align*}

All in all, we find the expression \eqref{eq:explicitFourier} we announced.

\end{document}